\documentclass[10pt,a4paper,oneside]{amsart}

\usepackage[T1]{fontenc}

\usepackage[francais]{babel}

\usepackage[french]{varioref}

\usepackage{latexsym}

\usepackage{amsfonts}

\usepackage{euscript}

\usepackage{amscd}

\DeclareMathSymbol{\varkappa}{\mathord}{AMSb}{"7B}

\newcommand{\function}[1]{\mathop{\mathrm{#1}}\nolimits}
\newcommand{\mathcan}[1]{\mathbb{#1}}

\newcommand{\Ad}{\mathcan{A}}
\newcommand{\CC}{\mathcan{C}}

\newcommand{\QQ}{\mathcan{Q}}
\newcommand{\RR}{\mathcan{R}}
\newcommand{\RRm}{\mathcan{R}^{\times}_{+}}

\newcommand{\ZZ}{\mathcan{Z}}

\newcommand{\Abs}[1]{\Bigl\lvert\,#1\,\Bigr\rvert}
\newcommand{\abs}[1]{\lvert#1\rvert}

\newcommand{\Demi}{\tfrac{1}{2}}
\newcommand{\idest}{\textit{i.~e. }}
\newcommand{\isom}{\function{\stackrel{\sim}{\longrightarrow}}}

\newcommand{\N}[1]{||#1||}
\newcommand{\Sch}[1]{\mathbf{#1}}
\newcommand{\set}[2]{\left\{#1\,\mid\,#2\right\}}

\renewcommand{\Re}{\mathop{\mathrm{Re}}\nolimits}
\renewcommand{\Im}{\mathop{\mathrm{Im}}\nolimits}

\newcommand{\cark}{\omega}
\newcommand{\dint}{\int_{\mathcan{R}}}
\newcommand{\EisCan}{\mathbf{E}}
\newcommand{\Eis}[1]{E(#1)}
\newcommand{\Esp}[1]{\mathsf{W}(#1)}
\newcommand{\EspScal}[2]{W(#1,#2)}
\newcommand{\Hecke}{H}
\newcommand{\RiemXi}{\Lambda}
\newcommand{\Toe}{E}
\newcommand{\Tor}{T(\chi)}
\newcommand{\Th}[1]{\theta(#1)}
\newcommand{\Zero}{\{0\}}

\DeclareMathOperator{\Cl}{Cl}
\DeclareMathOperator{\disc}{disc}
\DeclareMathOperator{\Gal}{Gal}
\DeclareMathOperator{\GL}{GL}

\DeclareMathOperator{\Nrm}{N}
\DeclareMathOperator{\Orth}{\mathbf{O}}

\DeclareMathOperator{\res}{res}
\DeclareMathOperator{\SL}{SL}
\DeclareMathOperator{\Spec}{Spec}
\DeclareMathOperator{\Supp}{Supp}
\DeclareMathOperator{\Tr}{Tr}
\DeclareMathOperator{\Tate}{\Delta}
\DeclareMathOperator{\vol}{vol}

\numberwithin{equation}{section}
\renewcommand{\theequation}{\arabic{section}.\arabic{equation}} %

\swapnumbers
\newtheorem{theorem}{Th\'eor\`eme}[section]
\newtheorem{proposition}[theorem]{Proposition}
\newtheorem{lemma}[theorem]{Lemme}
\newtheorem{corollary}[theorem]{Corollaire}

\newtheorem*{theorem*}{Th\'eor\`eme}
\newtheorem*{proposition*}{Proposition}
\newtheorem*{lemma*}{Lemme}
\newtheorem*{corollary*}{Corollaire}
\newtheorem*{definition*}{D\'efinition}
\newtheorem*{definitions*}{D\'efinitions}

\theoremstyle{remark}
\newtheorem{remark}[theorem]{Remarque}

\newtheorem*{remark*}{Remarque}
\newtheorem*{remarks*}{Remarques}
\newtheorem*{example*}{Exemple}
\newtheorem*{examples*}{Exemples}

\renewcommand{\theenumi}{\alph{enumi}}

\begin{document}

\selectlanguage{francais}

\title{Z\'eros des fonctions $L$ et Formes toro\"\i dales}

\author{Gilles Lachaud}
\address{Institut de Math\'ematiques de Luminy
\newline \indent
CNRS
\newline \indent
Luminy Case 907, 13288 Marseille Cedex 9 - FRANCE}
\email{lachaud@univmed.fr}




\begin{abstract}
\`A partir d'un corps de nombres $K$ de degr\'e $n$, on d\'efinit un tore
maximal $T$ de $G$ = $GL_{n}$. Si $\chi$ est un caract\`ere du groupe des
classes d'id\`eles de $K$, satisfaisant des conditions ad\'equates, les
formes toro\"\i dales pour $\chi$ sont les fonctions sur $G(\mathbf{Q})
Z(\mathbf{A}) \backslash G(\mathbf{A})$, dont le coefficient de
Fourier correspondant \`a $\chi$ par rapport au sous-groupe
induit par $T$ est nul. L'hypoth\`ese de Riemann pour $L(s, \chi)$ est
\'equivalente \`a des conditions portant sur certains espaces de formes
toro\"\i dales, construits \`a partir des s\'eries d'Eisenstein. Enfin, on
construit un espace de Hilbert et un op\'erateur auto-adjoint sur cet
espace, dont le spectre est \'egal \`a l'ensemble des z\'eros de $L(s, \chi)$
sur la droite critique.

\bigskip

\noindent\textsc{Abstract.}
An algebraic number field $K$ defines a maximal torus $T$ of the linear group $G = GL_{n}$. Let $\chi$ be a character of the idele class group of $K$, satisfying suitable assumptions. The $\chi$-toroidal forms are the functions defined on $G(\mathbf{Q}) Z(\mathbf{A}) \backslash G(\mathbf{A})$ such that the Fourier coefficient corresponding to $\chi$ with respect to the subgroup induced by $T$ is zero. The Riemann hypothesis is equivalent to certain conditions concerning some spaces of toroidal forms, constructed from Eisenstein series. Furthermore, we define a Hilbert space and a self-adjoint operator on this space, whose spectrum equals the set of zeroes of $L(s, \chi)$ on the critical line.
\end{abstract}


\maketitle

\setcounter{tocdepth}{2}
\tableofcontents

\section*{Introduction}

Pour \'etudier la r\'epartition des z\'eros de la fonction z\^eta de
Riemann, George P\'olya, et avant lui David Hilbert, d'apr\`es certains
t\'emoignages, ont sugg\'er\'e qu'il serait judicieux de trouver un espace
de Hilbert $\EuScript{H}$ et un op\'erateur $D$ dans $\EuScript{H}$, dont le
spectre est donn\'e par les z\'eros non triviaux de cette fonction, ou plus
pr\'ecis\'ement les nombres complexes $\gamma$ tels que
\begin{equation}
\label{DefZero}
\zeta(\Demi + i \gamma) = 0, \quad \abs{\Im(\gamma)} \leq \Demi.
\end{equation}
Par construction, les z\'eros non triviaux de la fonction z\^eta seront tous
sur la droite critique $\Re(s) = 1/2$ si et seulement si l'op\'erateur $D$
est auto-adjoint.

Il est facile de donner un exemple simple d'un tel espace: notons
$\EuScript{H}_{0}$ l'espace des sommes de puissances de la forme
\begin{equation}
\label{DefSomme}
f(x) = \sum_{\gamma} c_{\gamma} \, x^{i \gamma},
\end{equation}
o\`u $\gamma$ parcourt l'ensemble des nombres v\'erifiant
\eqref{DefZero}, et o\`u les $c_{\gamma}$ non nuls sont en nombre fini.
C'est un espace pr\'ehilbertien, muni de la norme
$$
\N{f}^{2} = \sum_{\gamma} \ \abs{c_{\gamma}}^{2}.
$$
Si $\EuScript{H}$ est l'espace de Hilbert correspondant, l'op\'erateur
diff\'erentiel
$$
D = - i \frac{d}{dx}
$$
d\'efinit un unique op\'erateur de $\EuScript{H}$, encore not\'e $D$, 
satisfaisant aux conditions requ\-ises. Si on remplace $\EuScript{H}$ par le
sous-espace engendr\'e par les sommes \eqref{DefSomme}, o\`u on exige que
$\gamma$ soit \emph{r\'eel}, l'op\'erateur diff\'erentiel $D$ d\'efinit un
op\'erateur auto-adjoint.

Il y a de nombreuses variantes de cette construction : par exemple, pour
tenir compte des multiplicit\'es des z\'eros, il convient d'introduire les
d\'eriv\'ees des puissances.

Enfin, on peut \'evidemment remplacer la fonction z\^eta de Riemann par la
fonction $L(s,\chi)$ d\'efinie par un caract\`ere $\chi$ du groupe
$C_{K}$ des classes d'id\`eles d'un corps global $K$, et c\oe tera.

Mais ce proc\'ed\'e de construction est \emph{tautologique}, car il utilise
explicitement les z\'eros de la fonction $\zeta(s)$, ou des fonctions
$L(s,\chi)$. A. Connes d\'efinit dans \cite{Connes3}
un \og espace de P\'olya-Hilbert
\fg{} comme un couple $(\EuScript{H}, D)$ :
\begin{itemize}
\item
form\'e d'un espace de Hilbert $\EuScript{H}$ et d'un
op\'erateur $D$ dans $\EuScript{H}$, ferm\'e, non born\'e, \`a domaine dense ; 
\item
tel que l'on ait
$$\Spec D = \set{\gamma \in \RR}{L(\Demi + i \gamma, \chi) = 0} \, ;$$
\item
d\'efini \emph{de mani\`ere intrins\`eque}, autrement dit d'une fa\c{c}on
qui n'utilise pas les fonctions $L(s,\chi)$.
\end{itemize}
Les espaces de P\'olya-Hilbert qu'il a construits dans \cite{Connes1},
\cite{Connes2}, \cite{Connes3} sont des espaces $L^{2}$ sur le groupe
$C_{K}$ ; ces constructions ont \'et\'e g\'en\'eralis\'ees par C. Soul\'e
\cite{Soule} aux fonctions $L$ automorphes. A. Connes signale
qu'il serait souhaitable de clarifier les rapports entre ces espaces de
P\'olya-Hilbert et \emph{l'espace des formes toro\"\i dales} introduit par Don
Zagier \cite{Zagier} ;  c'est ce que nous allons faire ici.

Nous construisons un espace de P\'olya-Hilbert \`a partir de formes
modulaires, c'est-\`a-dire de fonctions d\'efinies sur l'espace
$G(\QQ) Z(\Ad)
\backslash G(\Ad)$, o\`u $G = \mathbf{GL}_{n}$ est le groupe lin\'eaire
g\'en\'eral, et o\`u $\Ad$ est l'anneau des ad\`eles de $\QQ$. Les formes
modulaires que nous consid\'erons sont des familles de combinaisons
lin\'eaires de s\'eries d'Eisenstein, ou \emph{trains d'ondes
(wave-packets) d'Eisenstein}, qui s'\'ecrivent
$$F(g) = \sum_{s} \, a_{s} \, E(g,s),$$
et l'ensemble des points $s$ tels que $a_{s} \neq 0$ est le spectre de $F$.
Ce sont les analogues des polyn\^omes trigonom\'etriques. \'Etant donn\'e
une extension $K$ de degr\'e $n$ d'un corps de nombres alg\'ebriques $k$,
nous rappelons dans la section \ref{sec_Periodes} comment une base
fondamentale de $K$ d\'efinit une repr\'esentation alg\'ebrique $\pi$ du
sch\'ema $\mathcal{T}_{K/k}$ repr\'esentant le groupe multiplicatif de $K$
dans l'espace affine, dont l'image est un tore maximal $T$ d\'efini
sur $k$ ; le groupe
$\Ad_{K}^{\times}$ des id\`eles de
$K$ s'iden\-tifie au sous-groupe $T(\Ad_{k}) \subset G(\Ad_{k})$. Dans la
th\'eorie des formes modulaires, les formes paraboliques sont celles dont
l'int\'egrale sur l'image $N(k) \backslash N(\Ad_{k})$ d'un sous-groupe
unipotent $N$ est nul ; par analogie, si
$\chi$ est un caract\`ere du groupe $C_{K} = \Ad_{K}^{\times}/K^{\times}$des classes
d'id\`eles  (\emph{Gr\"{o}ssencharaktere de Hecke}), les \emph{formes
toro\"\i dales en
$\chi$} sont les formes modulaires dont \og l'int\'egrale p\'eriodique
\fg{} sur $T(k) Z(\Ad) \backslash T(\Ad_{k})$ est nulle :
$$
\int_{T(k) Z(\Ad) \backslash T(\Ad_{k})} F(h g) \,
\chi _{\circ} \pi^{-1}(h) \, dh = 0 \quad \text{pour} \ g \in G(\Ad),
$$
Supposons que $\chi$ soit non ramifi\'e en toute place de $K$, constant sur
$\Ad_{k}^{\times}$ et sur le sous-groupe compact maximal de
$K_{\infty}^{\times}$. En nous appuyant sur \emph{la formule de Hecke
ad\'elique} pour les s\'eries d'Eisenstein normalis\'ees (\'enonc\'ee dans
la section \ref{sec_Eisenstein}), nous \'etablissons que l'espace des
trains d'ondes d'Eisenstein qui sont des formes toro\"\i dales est un espace de
P\'olya-Hilbert. Plus pr\'ecis\'ement :

--- Dans la section \ref{sec_TrainsOndes}, nous montrons (th\'eor\`eme
\ref{Thm_CritTrainTor}) qu'un train d'ondes est toro\"\i dal en $\chi$ si et
seulement si son spectre est contenu dans le lieu des z\'eros de $L(s,
\chi)$. Il s'ensuit (corollaire \ref{Cor_EquivRiemann}) que l'hypoth\`ese de
Riemann pour $L(s,\chi)$ est \'equivalente \`a l'assertion suivante :

\emph{Tout train d'ondes d'Eisenstein toro\"\i dal a un spectre contenu dans
la droite critique}.

--- Dans la section \ref{sec_MaassSelberg}, nous donnons une autre condition
\'equivalente si $K$ est quadratique (corollaire \ref{Cor_EquivRiemannPp}) :

\emph{Tout train d'ondes d'Eisenstein toro\"\i dal est de carr\'e int\'egrable
en moyenne}.

--- Nous construisons dans la section \ref{sec_PolyaHilbert} un espace de
Hilbert $T^{2}_{\chi}(X)$ et un op\'erateur auto-adjoint $D_{\chi}$ dans
$T^{2}_{\chi}(X)$ tel que l'on ait (th\'eor\`eme \ref{Thm_PoHil})
$$
\Spec D_{\chi} =
\set{\lambda}{\lambda = \frac{1}{4} + \gamma^{2}, \quad
\gamma \in \RR, \quad L(\Demi + i\gamma, \chi) = 0 \,}.
$$
--- Nous exprimons la trace de certains op\'erateurs int\'egraux
comme une somme sur les z\'eros de $L(s,\chi)$ (corollaire
\ref{Cor_FormTrace}).

Enfin, dans l'appendice, nous faisons le lien entre le point de d\'epart de
la th\'eorie de Connes, introduite dans \cite{Connes3} et reprise dans le
th\'eor\`eme \ref{Thm_Connes}, et les trains d'ondes d'Eisenstein
(corollaire \ref{Cor_Isom}).

Les r\'esultats de ce travail ont \'et\'e annonc\'es dans \cite{Lachaud1} et \cite{Lachaud2}. Les formes toro\"\i dales y sont appel\'ees ``formes toriques'' ; nous nous sommes ici conform\'es \`a l'usage.

\newpage

\section{Int\'egrales p\'eriodiques et formule de Hecke}
\label{sec_Periodes}

\medskip

\subsection*{Repr\'esentations alg\'ebriques}~\medskip

Soient $k$ un corps de nombres alg\'ebriques, et $K$ une extension de
degr\'e $n$ de $k$. L'\og enveloppe alg\'ebrique \fg{}
de $K$ est le $k$-sch\'ema en $k$-alg\`ebres $\mathcal{A}_{K/k}$ tel que
l'on ait
$$
\mathcal{A}_{K/k}(k') = K \otimes_{k} k'
$$
pour
toute $k$-alg\`ebre $k'$. On dispose de morphismes de norme et de trace
$$
\Nrm_{K/k} : \mathcal{A}_{K/k} \longrightarrow \Sch{A}^{1}_{k}, \quad
\Tr_{K/k} : \mathcal{A}_{K/k} \longrightarrow \Sch{A}^{1}_{k}.
$$
Le $k$-sch\'ema en groupes repr\'esentant le groupe multiplicatif de
de $\mathcal{A}_{K/k}$ est
$$
\mathcal{T}_{K/k} = \set{(u, v) \in \Sch{A}^{n}_{k} \times
\Sch{G}_{m}} {\Sch{N}_{K/k}(u) = v}.
$$
Si $k'$ est une $k$-alg\`ebre, on a
$$
\mathcal{T}_{K/k}(k') = (K \otimes_{k} k')^{\times},
$$
ce qui montre que $\mathcal{T}_{K/k}$ est aussi le $k$-sch\'ema
$R_{K/k}(\Sch{G}_{m,K})$ obtenu en appliquant le foncteur de Weil $R_{K/k}$
de changement de base par restriction des scalaires au $K$-groupe
multiplicatif $\Sch{G}_{m,K}$. Le groupe alg\'ebrique
$\mathcal{T}_{K/k}$ est un tore de dimension $n$ d\'efini sur $k$.
Soient $\mathfrak{o}$ l'anneau des entiers de $k$, et $\mathfrak{O}$
l'anneau des entiers de $K$. On suppose qu'il existe une \emph{base
fondamentale} de $K$ sur $k$, c'est-\`a-dire une base
$\boldsymbol{\alpha} = (\alpha_{1}, \dots, \alpha_{n})$ du
$\mathfrak{o}$-module $\mathfrak{O}$, avec $\alpha_{1} = 1$. Si
$\mathfrak{o}$ est principal, toute extension de degr\'e fini de $k$ admet
une base fondamentale. La base $\boldsymbol{\alpha}$ fournit un isomorphisme
$\iota : \Sch{A}^{n}_{k} \longrightarrow \mathcal{A}_{K/k}$
donn\'e par
$$
\iota(u_{1}, \dots, u_{n}) = u_{1}\alpha_{1} + \dots + u_{n}\alpha_{n}.
$$
Le morphisme $\iota^{-1}$ induit un isomorphisme de 
$\mathcal{T}_{K/k}$ sur un ouvert de $\Sch{A}^{n}_{k}$.

Une \emph{repr\'esentation alg\'ebrique lin\'eaire} d'un $k$-sch\'ema en
alg\`ebres (resp. en groupes) dans $\Sch{A}^{n}_{k}$ est un $k$-morphisme
d'alg\`ebres (resp. de groupes) de ce sch\'ema dans le $k$-sch\'ema
$\Sch{M}_{n}$ des matrices carr\'ees d'ordre $n$ (resp. dans
$\mathbf{GL}_{n}$).  On d\'efinit la \emph{repr\'esentation r\'eguli\`ere
droite}
$\pi$ de
$\mathcal{A}_{K/k}$ dans $\Sch{A}^{n}_{k}$ de la mani\`ere suivante. Si
$k'$ est une $k$-alg\`ebre et si $\xi$ est un \'el\'ement de l'alg\`ebre
$\mathcal{A}_{K/k}(k')$, on note $\rho(\xi)$ la multiplication par $\xi$
dans $\mathcal{A}_{K/k}(k')$, et on pose
$$
{^{t}\!\pi}(\xi).u = \iota^{-1} \circ \rho(\xi) \circ \iota(u), \quad u \in
{k}'^{n},
$$
autrement dit
\begin{equation*}
\label{operation}
\iota(^{t}\!\pi(\xi).u) = \xi \, \iota(u), \quad u \in {k'}^{n}.
\end{equation*}
On a $\det \pi(\xi) = \Sch{N}_{K/k}(\xi)$. Si $e_{1}$ est le premier vecteur
de la base canonique de $\Sch{A}^{n}_{k}$, on a
\begin{equation*}
\label{InvPsi}
\iota^{-1}(\xi) = ^{t}\!\pi(\xi).e_{1},
\end{equation*}
ce qui montre que la repr\'esentation $\pi$ est fid\`ele. Si on note
$\boldsymbol{\omega} = (\omega_{1}, \ldots ,
\omega_{n})$ la \emph{base duale} de la base $\boldsymbol{\alpha}$, on a
aussi
$$
\iota^{-1}(\xi) =
(\Tr_{K/k}(\xi \omega_{1}), \dots, \Tr_{K/k}(\xi \omega_{n})).
$$
On en d\'eduit que les coefficients de $\pi$ sont donn\'es par
\begin{equation*}
\label{Coeff1}
\pi(\xi)_{ij} = \Tr_{K/k}(\xi \alpha_{i} \omega_{j}).
\end{equation*}
L'image $T$ de $\mathcal{T}_{K/k}$ par la repr\'esentation $\pi$
est un groupe lin\'eaire alg\'ebrique d\'efini sur $k$, de dimension $n$,
qui est un tore maximal de $G$. On a donc un isomorphisme
$$\pi : \mathcal{T}_{K/k} \isom T \subset G.$$
Le \emph{groupe projectif de} $K$ est le $k$-tore $\mathcal{P}_{K/k}$
d\'efini par la suite exacte
$$
\begin{CD}
1 & @>>> & \Sch{G}_{m,k} & @>>> & \mathcal{A}_{K/k}^{\times} & @>{N}>> &
\mathcal{P}_{K/k} & @>>> 1
\end{CD}
$$
Le morphisme $\iota^{-1}$ induit une immersion ouverte de $\mathcal{P}_{K/k}$
dans $\Sch{P}^{n - 1}_{k}$. D'autre part $\mathcal{P}_{K/k}(k) = k^{\times}
\backslash K^{\times}$. Le centre $Z$ de $G$ est inclus dans $T$ ; si
on pose $S = Z \backslash T$, on a un isomorphisme
$\pi : \mathcal{P}_{K/k} \isom S \subset PGL_{n}$.

\medskip

\subsection*{Sous-groupes compacts maximaux}~

\medskip

Le groupe $T(\mathfrak{o}) = T(k) \cap GL_{n}(\mathfrak{o})$ s'appelle le
\emph{groupe des unit\'es} de $T(k)$. Puisque $\boldsymbol{\alpha}$ est
une base fondamentale de $K$ sur $k$, on a $T(\mathfrak{o}) =
\pi(\mathfrak{O}^{\times})$. Soit $v$ une place finie de $k$. On note
$\mathfrak{O}_{v}$ le sous-anneau compact maximal de $K \otimes k_{v}$, de
telle sorte qu'on a un isomorphisme
$$
\begin{CD}
\mathfrak{O}_{v} & @>\Phi_{v}>> & \prod_{w \mid v}
\mathfrak{O}_{w},
\end{CD}
\qquad
\mathfrak{O}_{v} = \mathfrak{O} \otimes_{\mathfrak{o}} \mathfrak{o}_{v} = 
\alpha_{1}\mathfrak{o}_{v} + \dots + \alpha_{n}\mathfrak{o}_{v}.
$$
Le \emph{groupe des unit\'es} de $T(k_{v})$ est
$$
T(\mathfrak{o}_{v}) = T(k_{v}) \cap GL_{n}(\mathfrak{o}_{v}).
$$
Il est facile de voir que
$$
T(\mathfrak{o}_{v}) =
\set{g \in T(k_{v})}{g \mathfrak{o}^{n} = \mathfrak{o}^{n}} =
\set{g \in T(k_{v}) \cap M_{n}(\mathfrak{o}_{v})}{\det{g}
\in \mathfrak{o}_{v}^{\times}}.
$$
Le groupe $T(\mathfrak{o}_{v})$ est le sous-groupe compact maximal de
$T(k_{v})$. La repr\'esentation $\pi$ induit un isomorphisme
$\mathfrak{O}_{v}^{\times} \isom T(\mathfrak{o}_{v})$.

Supposons que $v$ soit une place infinie de $k$. Pour toute place $w$ de $K$
au dessus de $v$, on pose
$$
\xi^{(w)} = \alpha_{1}^{(w)} u_{1} + \dots + \alpha_{n}^{(w)} u_{n} 
\in K_{w} \quad  \text{si} \quad
\xi = \alpha_{1} \otimes u_{1} + \dots + \alpha_{n} \otimes u_{n}
\in K \otimes k_{v}.
$$
On pose $k_{\infty} = k \otimes \RR$. L'unique sous-groupe compact maximal
de
$$
K_{\infty}^{\times} = (K \otimes k_{\infty})^{\times} \cong \prod_{w \mid
\infty} K_{w}^{\times}
$$
est
$$
U_{\infty} = \prod_{v \, \infty} \set{\xi \in
K_{\infty}^{\times}}{\abs{\xi^{(w)}}_{w} = 1
\ \text{si} \ w \mid v}
$$
Si on note $U_{\infty}(T)$ le sous-groupe compact maximal de
$T(k_{\infty})$, la repr\'esentation $\pi$ induit des isomorphismes
$K_{\infty}^{\times} \isom T(k_{\infty})$ et $U_{\infty} \isom
U_{\infty}(T)$. 

\begin{lemma}
\label{DecSpecPi}
Supposons $k = \QQ$. Si $\xi \in K \otimes \RR$, on a $\pi(\xi)
= A\pi_{0}(\xi)A^{-1}$, o\`u
$$
A =
\begin{bmatrix}
\alpha_{1}^{(1)} & \dots & \alpha_{n}^{(1)} \\
\dots           & \dots & \dots \\
\alpha_{1}^{(n)} & \dots & \alpha_{n}^{(n)}
\end{bmatrix} ,
\qquad
\pi_{0}(\xi) =
\begin{bmatrix}
\xi^{(1)} & 0         & \dots & 0     \\
0         & \xi^{(2)} & \dots & 0     \\
\dots     & \dots     & \dots & \dots \\
0         & \dots     & 0     & \xi^{(n)} \\
\end{bmatrix} ,
$$
o\`u $\xi^{(1)}, \dots, \xi^{(n)}$ sont les $n$ isomorphismes distincts de
$K$ dans $\bar{\QQ}$.
\end{lemma}

\begin{proof}
Si $1 \leq i \leq n$, on pose $a^{(i)} = Ae_{i} = (\alpha_{1}^{(i)}, \dots,
\alpha_{n}^{(i)})$. Si $u \in \RR^{n}$, on a $\iota(^{t}\!\pi(\xi)u) = \xi
\, \iota(u)$ et $\iota(u) =  {^{t}\!u}.a^{(1)}$ par d\'efinition, et par
cons\'equent
$$
\xi \, {^{t}\!u}.a^{(1)} =
\xi \, \iota(u) = \iota(^{t}\!\pi(\xi)u) = {^{t}\!u}.\pi(\xi)a^{(1)}.
$$
On en d\'eduit $\pi(\xi)a^{(i)} = \xi^{(i)} a^{(i)}$.
D'autre part
$$
\pi(\xi)Ae_{i} = \pi(\xi)a^{(i)} = \xi^{(i)} \, a^{(i)} = \xi^{(i)} \,
Ae_{i} = A\pi_{0}(\xi)e_{i},
$$
ce qui prouve que $\pi(\xi)A = A\pi_{0}(\xi)$.
\end{proof}

La matrice
\begin{equation}
\label{MatSymK}
p_{_{K}} = A.^{t}\!\bar{A} =
\begin{bmatrix}
\Tr \alpha_{1} \overline{\alpha}_{1} & \dots & \Tr \alpha_{1}
\overline{\alpha}_{n} \\
\dots                    & \dots & \dots \\
\Tr \alpha_{n} \overline{\alpha}_{1} & \dots & \Tr \alpha_{n}
\overline{\alpha}_{n}
\end{bmatrix}
\end{equation}
est sym\'etrique d\'efinie positive \`a coefficients dans $\ZZ$,  de
d\'eterminant $\disc K$. On note
$$
\mathbf{O}_{n}(\RR) = \set{g \in GL_{n}(\RR)}{^{t}\!g.g = \mathbf{1}_{n}}
$$
le sous-groupe compact maximal usuel de $GL_{n}(\RR)$.
\begin{proposition}
\label{SsGpeCompT}
Supposons $k = \QQ$. Soit $q_{_{K}} \in G(\RR)$ telle que
$q_{_{K}}.{^{t}\!q_{_{K}}} = P$. Alors
$$
U_{\infty}(T) = T(\RR) \cap q_{_{K}} \mathbf{O}_{n}(\RR) \, q_{_{K}}^{-1}.
$$
et ce groupe est isomorphe \`a $\{ \pm 1 \}^{r_{1}} \times
\mathbf{SO}_{2}(\RR)^{r_{2}}$.
\end{proposition}

\begin{proof}
Si $\xi \in K \otimes \RR$, on a $\pi(\xi).A = A.\pi_{0}(\xi)$ par le
lemme \ref{DecSpecPi}. On en d\'eduit que ${^{t}\!\bar{A}}.{^{t}\!\pi(\xi)}
= \pi_{0}(\bar{\xi}).{^{t}\!\bar{A}}$ puisque $\pi(\xi) \in G(\RR)$, et
$$
A.\pi_{0}(\xi \, \bar{\xi}).{^{t}\!\bar{A}} =
\pi(\xi).P.{^{t}\!\pi(\xi)}.
$$
Posons maintenant $\theta(\xi) = q_{_{K}}^{-1}\pi(\xi)q_{_{K}}$. On a
$$
q_{_{K}}.\theta(\xi).^{t}\!\theta(\xi).^{t}\!q_{_{K}} =
\pi(\xi).P.{^{t}\!\pi(\xi)}
$$
Le sous-groupe compact maximal $U_{\infty}$ de $(K \otimes \RR)^{\times}$
est isomorphe au produit de groupes de l'\'enonc\'e. D'autre part
$\pi$ d\'efinit un isomorphisme $U_{\infty} \rightarrow
U_{\infty}(T)$.  Puisque  $\xi \in U_{\infty}$  si et seulement si
$\pi_{0}(\xi \, \bar{\xi}) = \mathbf{1}_{n}$, ceci d\'emontre que $\pi(\xi)
\in U_{\infty}(T)$ si et seulement si $\theta(\xi) \in \mathbf{O}_{n}(\RR)$.
\end{proof}

\subsection*{Groupes de classes d'id\`eles} ~\medskip

On note respectivement $\Ad_{K}^{~}$ et $\Ad_{K}^{\times}$ l'anneau des
ad\`eles et le groupe des id\`eles de $K$, et $C_{K} = {K}^{\times}
\backslash \Ad_{K}^{\times}$ le \emph{groupe des classes d'id\`eles} de $K$.
La repr\'esentation $\pi$ induit des isomorphismes
$$
\begin{CD}
C_{K} & @>{\sim}>> & T(k) \backslash T(\Ad_{k}),
\quad
\Ad_{k}^{\times} \backslash \Ad_{K}^{\times} & @>{\sim}>> & S(\Ad_{k}),
\quad
\end{CD}
$$
et aussi un isomorphisme
$$
\begin{CD}
C_{k} \backslash C_{K} & @>{\sim}>> & S(k) \backslash S(\Ad_{k}).
\end{CD}
$$
Pour $\xi \in \Ad_{K}$, on a
$$
\abs{\det \pi(\xi)}_{\Ad_{k}} = \abs{N_{K/k}(\xi)}_{\Ad_{k}} =
\abs{\xi}_{\Ad_{K}}.
$$
Le groupe ${K}^{\times}$ est un sous-groupe discret de
$$
\Ad^{1}_{K} = \set{\xi \in \Ad^{\times}_{K}}{\abs{\xi}_{\Ad_{K}} = 1},
$$
et le quotient $C_{K}^{1} = {K}^{\times} \backslash \Ad_{K}^{1}$ est compact.
Pour tout $x \in \RRm$, on note $\xi(x) = (\xi_{v}(x))$ l'id\`ele de
$\Ad^{\times}_{K}$ tel que $\xi_{v}(x) = 1$ pour toute place finie de $K$ et
tel que $\xi_{v}(x) = x$ pour toute place infinie de $K$. Alors $x \mapsto
\xi(x)$ induit un isomorphisme de $\RRm$ sur un sous-groupe $N$ du groupe
$C_{K}$ et $C_{K}$ est produit direct de $N$ et de $C_{K}^{1}$ \cite[Cor. 2,
p. 76]{BNT}. On en d\'eduit que
$$
C_{k} \backslash C_{K} = C^{1}_{k} \backslash C^{1}_{K} =
K^{\times} \Ad^{1}_{k} \backslash \Ad^{1}_{K},
$$
et que $C_{k} \backslash C_{K}$ est compact ; il en va de m\^eme du groupe
$S(k) \backslash S(\Ad_{k})$. La repr\'e\-sen\-ta\-tion $\pi$ d\'efinit un
isomorphisme de $\Ad^{1}_{K}$ sur le sous-groupe ferm\'e
$$T^{1}(\Ad_{k}) = \set{h \in T(\Ad_{k})}{\abs{\det h}_{\Ad_{k}} = 1}$$
et de $N$ sur $Z(\RR)$. Le groupe $T(\Ad_{k})$ est le produit direct de
$T^{1}(\Ad_{k})$ et de $Z(\RR)$, et $T(k)$ est un sous-groupe discret de
$T^{1}(\Ad_{k})$, \`a quotient compact. On a un isomorphisme
$$
S(k) \backslash S(\Ad_{k})
\cong T(k) Z(\Ad) \backslash T(\Ad)
\cong T^{1}(k) Z^{1}(\Ad) \backslash T^{1}(\Ad).
$$
Le \emph{groupe des classes} de $K$ \cite[p. 87]{BNT} est
$$
\Cl(K) =
K^{\times} \backslash \Ad_{K}^{\times} / 
K_{\infty}^{\times} \, \widehat{\mathfrak{O}}^{\times},
\quad \text{o\`u} \quad
\widehat{\mathfrak{O}}^{\times} = \prod_{w} \mathfrak{O}_{w}^{\times}.
$$
C'est un groupe fini. La repr\'e\-sen\-ta\-tion $\pi$ d\'efinit un
isomorphisme
$$
\Cl(k) \backslash \Cl(K) \isom
T(k)  Z(\Ad) \backslash T(\Ad) / T(k_{\infty})T(\widehat{\mathfrak{o}})
\quad \text{o\`u} \quad
T(\widehat{\mathfrak{o}}) = \prod_{v} T(\mathfrak{o}_{v}).
$$
et $S(k) \backslash S(\Ad_{k})$ est extension de
$\Cl(k)\backslash \Cl(K)$ par le tore r\'eel $T(\mathfrak{o}) Z(\RR)
\backslash T(k_{\infty})$.

Le sous-groupe compact maximal de $\Ad_{K}^{\times}$ est $U = U_{\infty} \,
\widehat{\mathfrak{O}}^{\times}$, et celui de $T(\Ad)$ est 
$$
U_{\Ad}(T) = U_{\infty}(T) \, T(\widehat{\mathfrak{o}}),
\quad \text{o\`u} \quad
T(\widehat{\mathfrak{o}}) = \prod_{v} T(\mathfrak{o}_{v}).
$$
Il y a aussi un isomorphisme
$$
K^{\times} \Ad_{k}^{\times} \backslash \Ad_{K}^{\times} / U
\isom
T(k) Z(\Ad_{k}) \backslash T(\Ad_{k}) / U_{\Ad}(T) =: Q,
$$
Si $k = \QQ$, le groupe compact $Q$ est extension de $\Cl(K)$ par
le \emph{tore de Dirichlet}
$$
T(\ZZ) Z(\RR) \backslash T(\RR) / U_{\infty}(T)
\cong
\mathbf{SO}_{2}(\RR)^{r},
$$
avec $r = r_{1} + r_{2} - 1$, si $K$ admet $r_{1}$ places r\'eelles et
$r_{2}$ places imaginaires, autrement dit on a une suite exacte
$$
\begin{CD}
1 @>>> \mathbf{SO}_{2}(\RR)^{r} @>>> Q @>>> \Cl(K) @>>> 1.
\end{CD}
$$
On note $\mathsf{X}(G)$ le groupe des caract\`eres d'un groupe localement
compact $G$. On note simplement $\mathsf{X}$ le groupe discret
$\mathsf{X}(Q)$ ; il s'identifie au sous-groupe des caract\`eres de
$\Ad_{K}^{\times}/K^{\times}$ (\emph{Gr\"{o}ssencharaktere de Hecke}), qui
sont constants sur $\widehat{\mathfrak{O}}^{\times}$, autrement dit
\emph{non ramifi\'es en toute place de $K$}, et aussi qui sont constants sur
$\Ad_{k}^{\times}$ et sur $U_{\infty}$. Si $k = \QQ$, on a une suite exacte
$$
\begin{CD}
1 @>>> \mathsf{X}(\Cl(K)) @>>> \mathsf{X} @>>> \ZZ^{r} @>>> 1,
\end{CD}
$$
le groupe $\mathsf{X}$ est extension de $\ZZ^{r}$ par le groupe
$\mathsf{X}(\Cl(K))$.

\medskip

\subsection*{Mesures sur les groupes d'id\`eles}~

\medskip

On choisit la mesure de Haar suivante sur $C_{K} = N \times C^{1}_{K}$ :
$$
\int_{C_{K}^{1}} d^{\times}\!\xi = 1,
$$
$$
\int_{N} f(\xi) \,d^{\times}\!\xi =
n \int_{0}^{\infty} f(\xi(x)) \, d^{\times}\!x =
\int_{0}^{\infty} f(\xi(x^{1/n})) \, d^{\times}\!x.
$$
On d\'efinit comme suit la mesure de Haar sur $T(k) \backslash T(\Ad)$ :
$$
\int_{T(k) \backslash T(\Ad)} F(h) \, dh =
\int_{C_{K}} F(\pi(\xi)) \, d^{\times}\!\xi.
$$
La mesure invariante sur $S(k) \backslash S(\Ad_{k})$ est d\'efinie
par
\begin{equation}
\label{MesInv}
\int_{T(k) \backslash T(\Ad)} F(h) \, dh =
\int_{S(k) \backslash S(\Ad_{k})} \, \int_{Z(k) \backslash Z(\Ad)}
F(z \dot{h}) \, dz \, d\dot{h}.
\end{equation}
De cette mani\`ere, le groupe $S(k) \backslash S(\Ad_{k})$ est de
volume $1$. Si $F$ est une fonction d\'efinie sur $T(\Ad)$, invariante \`a
gauche par $T(k)Z(\Ad)$, et suffisamment r\'eguli\`ere, on pose
$$
\oint F(h) \, dh = \int_{S(k) \backslash S(\Ad_{k})} F(h) \, dh =
\int_{T(k) Z(\Ad) \backslash T(\Ad)} F(h) \, dh.
$$

\medskip

\subsection*{Fonctions $L$ et distribution de Weil}~

\medskip

Si $\Phi$ appartient \`a l'espace $\EuScript{S}(\Ad_{k})$ des fonctions
standard sur $\Ad_{k}$, si $\chi$ est un caract\`ere de $C_{k}$, et si $s
\in \CC$, \emph{l'int\'egrale de Tate} de $\Phi$ est \cite[Eq. (4), p.
118]{BNT} :
\begin{equation}
\label{Weil}
\Tate_{k}(\Phi, s, \chi) =
\int_{\Ad_{k}^{\times}} \Phi(\xi) \, \chi(\xi) \, \abs{\xi}_{\Ad_{k}}^{s} \,
d^{\times}\!\xi.
\end{equation}
Soit $P$ l'ensemble des places finies de $k$ o\`u $\chi$ n'est pas
ramifi\'e. La s\'erie $L$ de Hecke attach\'ee \`a $\chi$ est
$$
L(s, \chi) = \prod_{\mathfrak{p} \in P}
\left( 1 - \dfrac{\chi(\mathfrak{p})}{N(\mathfrak{p})^{s}} \right)^{- 1}
$$
\cite[Eq.(11), p. 133]{BNT}. On sait, gr\^ace \`a A. Weil \cite[Eq. (6), p.
128]{BNT}, \cite[Lem. 1, p. 76]{Connes3}, que la fonction $\Tate_{k}(\Phi,
s, \chi)$ est m\'eromorphe dans $\CC$, et on a
\begin{equation}
\label{Tate}
\Tate_{k}(\Phi,s,\chi) = c_{k}^{-1} \, L(s,\chi) \, \Tate'_{k}(\Phi,s,\chi),
\end{equation}
o\`u $\Tate'_{k}(\Phi,s,\chi)$ est une forme lin\'eaire sur 
$\EuScript{S}(\Ad_{k})$ qui est holomorphe pour $\Re(s) > 0$. Enfin, $c_{k}
= 2^{r_{1}}(2 \pi)^{r_{2}} h R/e$, o\`u $r_{1}$ et $r_{2}$ sont
respectivement le nombre de places r\'eelles et imaginaires de $k$, o\`u
$h$ est son nombre de classes et $R$ son r\'egulateur, et o\`u $e$ est
l'ordre du groupe des racines de l'unit\'e dans $k$ \cite[p. 128]{BNT}.

Supposons maintenant $\chi \in \mathsf{X}$. On d\'efinit une fonction
$\Phi_{0}$ sur $\Ad_{K}^{\times}$ de la mani\`ere suivante \cite[p.
131]{BNT} :
\begin{itemize}
\item
aux places infinies : on suppose que le plongement $\xi \mapsto \xi^{(i)}$
est r\'eel pour $1 \leq i \leq r_{1}$, qu'il est complexe pour $r_{1} + 1
\leq i \leq 2r_{2}$, et que $\xi^{(r_{1} + r_{2} + i)} =
\overline{\xi^{(i)}}$ pour $1 \leq i \leq r_{2}$. On pose
$$
F(\xi) = \sum_{i = 1}^{r_{1}} \xi^{(i) 2} +
2 \sum_{i = r_{1} + 1}^{r_{1} + r_{2}} \abs{\xi^{(i)}}^{2}.
$$
Si $\xi \in K_{\infty}$, on prend $\Phi_{\infty}(\xi) = e^{- F(\xi)}$.
\item
aux places finies, $\chi_{v}$ est non ramifi\'e. On prend pour
$\Phi_{v}$ la fonction caract\'eristique de $\mathfrak{O}_{v}$.
\end{itemize}
Si $w$ est une place infinie de $K$, on a \cite[p. 130]{BNT}
$$
\chi_{w}(x) = \abs{x}_{w}^{i \rho_{w}} \quad
\text{pour tout} \ x \in K_{w}, \quad \text{o\`u} \ \rho_{w} \in \RR.
$$
On pose \cite[Lem. 8, p. 127]{BNT} :
$$
\Gamma_{\RR}(s) = \pi^{-s/2} \Gamma(s/2), \quad
\Gamma_{\CC}(s) = (2 \pi)^{1 - s} \Gamma(s),
$$
et on pose $\Gamma_{w} = \Gamma_{\RR}$ ou $\Gamma_{w} = \Gamma_{\CC}$
suivant que $w$ est r\'eelle ou complexe. On d\'eduit de \cite[Eq. (10), p.
133]{BNT} que l'on a
\begin{equation}
\label{DeltaPrime}
\Delta'_{K}(\Phi_{0},s,\chi) = \Gamma(s, \chi).
\end{equation}
o\`u
$$
\Gamma(s, \chi) = \prod_{w \, \infty} \Gamma_{w}(s + i \rho_{w}).
$$

\medskip

\subsection*{Int\'egrales p\'eriodiques et s\'eries de Fourier}~

\medskip

L'\emph{int\'egrale p\'eriodique}  de fr\'equence $\chi \in \mathsf{X}$
d'une fonction $F \in C(T(k) Z(\Ad_{k}) \backslash G(\Ad_{k}))$ est la
fonction
$\Pi_{\chi}(F)$ d\'efinie par
$$
\Pi_{\chi}(F)(g) = \oint F(h g) \, \chi_{\pi}(h) \, dh.
$$
o\`u on a pos\'e $\chi_{\pi}(h) = \chi _{\circ} \pi^{-1}(h)$ pour tout $h
\in T(\Ad)$. Puisque
$$
\Pi_{\chi}(F)(g) =
\int_{Q} \chi_{\pi}(\dot{h}) \, d\dot{h}
\int_{U_{Q}} F(\dot{h} \eta g) d\eta,
$$
o\`u $U_{Q}$ est l'image de $U_{\Ad}(T)$ dans $Q$, La fonction
$\Pi_{\chi}(F)(g)$ est le coefficient de Fourier de fr\'equence $\chi$ de
la fonction
$$\dot{F}(\dot{h}) = \int_{U_{Q}} F(\dot{h} \eta g) d\eta.$$
On a, dans $L^{2}(Q)$ tout au moins,
\begin{equation}
\label{IntSiegel1}
\int_{U_{Q}} F(\dot{h} \eta g) d\eta =
\sum_{\chi \in \mathsf{X}} \Pi_{\bar{\chi}}(F)(g) \, \chi_{\pi}(\dot{h}).
\end{equation}
Le sous-groupe maximal standard de $G(\Ad_{k})$ est 
$$\mathbf{K} = \prod_{v} \Orth_{v}(n),$$
o\`u $\Orth_{v}(n)$ est le sous-groupe maximal standard de $GL_{n}(k_{v})$
si $v$ est une place infinie et o\`u $\Orth_{v}(n) = GL_{n}
(\mathfrak{o}_{v})$ si $v$ est une place finie. Supposons $k = \QQ$. 
Consid\'erons la \emph{vari\'et\'e modulaire}
$$
X = G(k) Z(\Ad_{k}) \backslash G(\Ad_{k}) / \mathbf{K}.
$$
D'apr\`es la proposition \ref{SsGpeCompT}, il existe $q_{_{K}} \in G(\RR)$
tel que
$$
T(\RR) \cap q_{_{K}} \mathbf{O}_{n}(\RR) \, q_{_{K}}^{-1} = U_{\infty}(T).
$$
Si $g_{_{K}}$ est la matrice de $G(\Ad)$ telle que $(g_{_{K}})_{\infty} =
q_{_{K}}$ et $(g_{_{K}})_{v} = 1$ pour toute place finie $v$ de $k$, on a
$$
T(\Ad) \cap g_{_{K}} \mathbf{K} \, g_{_{K}}^{-1} = U_{\Ad}(T).
$$
Si $F$ est une fonction sur $X$ suffisamment r\'eguli\`ere, on a
$$
F(h \eta g_{_{K}}) =
F(hg_{_{K}}) \quad \text{pour tout} \ \eta \in U_{\Ad}(T),
$$
la fonction $h \mapsto F(hg_{_{K}})$ est d\'efinie sur $Q$, et on a
$$
\int_{U_{Q}} F(h \eta g_{_{K}}) d\eta = F(hg_{_{K}}),
$$
ce qui implique que pour $\chi \in \mathsf{X}$, on a 
$$
c_{\chi} := \Pi_{\bar{\chi}}(F)(g_{_{K}}) = \int_{Q} F(\dot{h}g_{_{K}}) \, 
\overline{\chi}_{\pi}(\dot{h}) \, d\dot{h} =
\oint F(h g_{_{K}}) \, \overline{\chi}_{\pi}(h) \, dh.
$$
On d\'eduit alors de \eqref{IntSiegel1} :
\begin{equation}
\label{IntSiegel2}
F(h g_{_{K}}) =
\sum_{\chi \in \mathsf{X}} c_{\chi} \, \chi_{\pi}(h).
\end{equation}
On retrouve une formule de Siegel : voir \eqref{IntSiegel3}.

\medskip

\subsection*{S\'eries d'Eisenstein g\'en\'erales et formule de Hecke}~

\medskip

On pose maintenant $\Ad = \Ad_{k}$ et $\abs{x} = \abs{x}_{\Ad_{k}}$ pour
$x \in \Ad$. Si $\varphi \in \EuScript{S}(\Ad^{n})$, on pose 
$$
\Th{\varphi}(g) = \abs{\det g}^{1/2}
\sum_{x \in k^{n} - \{0\}} \varphi(^{t}\!g.x)
\quad \text{pour} \quad g \in G(\Ad),
$$
Ces s\'eries sont not\'ees $E(\varphi)$ dans \cite{Connes3} ; il vaut mieux
changer les notations \`a cause des \emph{s\'eries d'Eisenstein
g\'en\'erales (de rang relatif un)}, qui sont les fonctions
\begin{equation}
\label{Eisenstein}
\Eis{\varphi}(g,s, \cark) = \int_{Z(k) \backslash Z(\Ad)} \cark(\det zg) \,
\abs{\det zg}_{\Ad_{k}}^{s - \Demi} \, \Th{\varphi}(zg) \, dz,
\end{equation}
o\`u $\cark$ est un caract\`ere de
$C_{k} = k^{\times} \backslash \Ad_{k}^{\times}$
et o\`u la mesure de Haar est d\'efinie par
\begin{equation}
\label{MesZ}
\int_{Z(k) \backslash Z(\Ad)} f(z) \, dz =
n \int_{k^{\times} \backslash \Ad^{\times}}
f(\lambda.\mathbf{1}_{n}) \, d^{\times}\!\lambda.
\end{equation}
Les s\'eries $\Th{\varphi}$ et $\Eis{\varphi}$ figurent d\'ej\`a dans
\cite{Godement1}, \cite{Godement2} lorsque $n = 2$. L'int\'egrale
$\Eis{\varphi}(g,s)$ converge si $\Re(s) > 1$ et d\'efinit une fonction sur
$G(k) Z(\Ad) \backslash G(\Ad)$. La fonction $s \mapsto \Eis{\varphi}(g,s)$
se prolonge en une fonction m\'eromorphe dans $\CC$. On prend comme mesure
de Haar sur $\Ad^{n}$ la mesure telle que $\vol k^{n} \backslash \Ad^{n} =
1$, et on note $\EuScript{F}\varphi$ la transform\'ee de Fourier additive
de $\varphi$ d\'efinie par cette mesure. Si $\cark^{n} \neq 1$, la fonction
$s \mapsto \Eis{\varphi}(g,s)$ est enti\`ere ; si $\cark^{n} = 1$, les seuls
p\^oles \'eventuels de la fonction $s \mapsto \Eis{\varphi}(g,s)$ sont les
points $0$ et $1$. Ces p\^oles sont simples, de r\'esidus
$$
\res_{s = 0} \Eis{\varphi}(g,s) = - \, \varphi(0) \, \cark(\det g),\quad
\res_{s = 1} \Eis{\varphi}(g,s) = + \, \EuScript{F}\varphi(0) \, \cark(\det g).
$$
La fonction $\Eis{\varphi}(g,s)$ satisfait \`a l'\'equation fonctionnelle
\begin{equation}
\label{EqFct1}
\Eis{\varphi}(g, s, \cark) = \Eis{\EuScript{F}\varphi}(^t\!g^{-1}, 1 - s, \bar{\cark}).
\end{equation}
\emph{La formule de Hecke}, sous la forme donn\'ee dans \cite{Zagier} (si $n
= 2$) et dans \cite{Wielonsky1}, \cite{Wielonsky2} (dans le cas g\'en\'eral)
dit que le coefficient de Fourier d'une s\'erie d'Eisenstein pour un
caract\`ere $\chi \in C_{k} \backslash C_{K}$ s'exprime \`a l'aide de la
\emph{fonction $L$ de Hecke} associ\'ee \`a $\chi$. On peut l'\'enoncer
ainsi :

\begin{proposition}
\label{Prop_Hecke}
Soient $\chi$ un caract\`ere de $C_{k} \backslash C_{K}$ et $\cark$ un
caract\`ere de $C_{k}$. Si $s \in \CC$ n'est pas un p\^ole de $\Eis{\varphi}(g,s,
\cark)$, si $g \in G(\Ad)$, si $a \in k^{n} - \{0\}$, et si $\varphi \in
\EuScript{S}(\Ad^{n})$, on a 
\begin{equation}
\label{Hecke1}
\oint \Eis{\varphi}(hg,s, \cark) \, \chi_{\pi}(h) \, dh =
\cark(\det g) \, \abs{\det g}_{\Ad_{k}}^{s} \,
\Tate_{K}(\Phi_{g,a},s, (\cark _{\circ} N) \, \chi),
\end{equation}
o\`u la fonction $\Phi_{g,a} \in \EuScript{S}(\Ad_{K})$ est d\'efinie par
$$
\Phi_{g,a}(\xi) = \varphi(^{t}\!g^{t}\!\pi(\xi).a)
\qquad (\xi \in \Ad_{K}).
$$
\end{proposition}

\begin{proof}
Si $\lambda \in \Ad^{\times}_{k}$, posons $\rho(\lambda) = \cark(\lambda) \,
\abs{\lambda}_{\Ad_{k}}^{s}$ ; c'est un quasi-caract\`ere de $\Ad^{\times}_{k}/
k^{\times}$. Si $\alpha = \iota(a)$, on a :
\begin{eqnarray*}
& &
\oint \Eis{\varphi}(hg,s, \cark) \, \chi_{\pi}(h) \, dh
\\ & = &
\int_{T(k) \backslash T(\Ad)} \chi_{\pi}(h) \sum_{x \in k^{n} - \{0\}}
\varphi(^{t}\!g^{t}\!h.x) \, \rho(\det hg) \, dh
\\ & = &
\rho(\det g) \int_{K^{\times} \backslash \Ad_{K}^{\times}}
\chi(\xi) \sum_{q \in K^{\times}}
\varphi(^{t}\!g^{t}\!\pi(\xi).\iota^{-1}(q)) \, \rho(N(\xi)) \, d^{\times}\!\xi
\\ & = &
\rho(\det g) \int_{K^{\times} \backslash \Ad_{K}^{\times}}
\chi(\xi) \sum_{q \in K^{\times}}
\varphi(^{t}\!g^{t}\!\pi(\xi).\iota^{-1}(\alpha q)) \, \rho(N(\xi)) \, d^{\times}\!\xi
\\  & = &
\rho(\det g) \int_{K^{\times} \backslash \Ad_{K}^{\times}}
\chi(\xi) \sum_{q \in K^{\times}}
\varphi(^{t}\!g^{t}\!\pi(\xi).^{t}\!\pi(q).a) \, \rho(N(\xi)) \,
d^{\times}\!\xi
\\  & = &
\rho(\det g) \int_{K^{\times} \backslash \Ad_{K}^{\times}}
\chi(\xi) \sum_{q \in K^{\times}}
\varphi(^{t}\!g^{t}\!\pi(q\xi).a) \, \rho(N(q\xi)) \, d^{\times}\!\xi
\\  & = &
\rho(\det g) \int_{\Ad_{K}^{\times}}
\Phi_{g,a}(\xi) \, \rho(N(\xi)) \, \chi(\xi) \, d^{\times}\!\xi
\\ & = &
\rho(\det g) \, \Tate_{K}(\Phi_{g,a},s, (\cark _{\circ} N) \, \chi),
\end{eqnarray*}
ce qui \'etablit le r\'esultat.
\end{proof}

\medskip

\section{S\'eries d'Eisenstein}
\label{sec_Eisenstein}

\medskip

\subsection*{S\'eries d'Eisenstein normalis\'ees} ~\medskip

\textit{On suppose maintenant $\cark = 1$ et $k = \QQ$.}

On note $P$ le sous-groupe parabolique maximal standard de $G$ de type
$(n - 1, 1)$ form\'e des matrices
$$
p =
\begin{pmatrix} g' & ^{t}\!x \\ 0 & t \\ \end{pmatrix}
$$
o\`u $g' \in \mathbf{GL}_{n - 1}$, o\`u $t \in \mathbf{GL}_{1}$, et
o\`u $x$ est une matrice ligne \`a $n - 1$ \'el\'ements. Le radical
unipotent de $P$ est
$$
N = \left\{ 
\begin{pmatrix} \mathbf{1}_{n - 1} & ^{t}\!x \\ 0 & 1 \\ \end{pmatrix}
\right\},
$$
et $P$ est le produit de $N$ par le sous-groupe de Levi
$$
M = \left\{ 
\begin{pmatrix} g' & 0 \\ 0 & t \\ \end{pmatrix}
\right\}.
$$
Le centre de $M$ est
$$
A = \left\{ 
\begin{pmatrix} t'.\mathbf{1}_{n - 1} & 0 \\ 0 & t \\ \end{pmatrix}
\right\},
$$
o\`u $t$ et $t'$ sont dans $\mathbf{GL}_{1}$. Si $p \in P$ est la matrice
ci-dessus, on pose $\alpha(p) = t$, de telle sorte que ${^{t}\!e_{n}}.p =
\alpha(p).{^{t}\!e_{n}}$. Le module de $P(\Ad)$ est
$$
\delta_{P}(p) = \delta_{P} 
\begin{pmatrix} g' & ^{t}\!x \\ 0 & t \\ \end{pmatrix} =
\Abs{\dfrac{\det p}{\alpha(p)^{n}}} =
\Abs{\dfrac{\det g'}{t^{n - 1}}}_{\Ad},
\quad p \in P(\Ad).
$$
On a $G(\Ad) = P(\Ad).\mathbf{K}$, o\`u $\mathbf{K}$ est le sous-groupe
maximal standard de $G(\Ad_{k})$. Si $g = p \kappa \in G(\Ad)$, avec
$p \in P$ et $\kappa \in \mathbf{K}$, on pose $\delta_{P}(g) =
\delta_{P}(p)$. Il n'y a pas d'ambigu\"{i}t\'e, car si $\kappa \in P(\Ad)
\cap \mathbf{K}$, alors $\alpha(\kappa)$ et $\det \kappa$ sont dans
$\Ad^{1}$, par suite $\delta_{P}(\kappa) = 1$. Si $p \in P(\Ad)$, si $g \in
G(\Ad)$ et si $\kappa \in \mathbf{K}$, on a
$$
\delta_{P}(p g \kappa) = \delta_{P}(p) \delta_{P}(g), \quad
\delta_{P}(p^{-1}) = \delta_{P}(p)^{-1}.
$$
On va s'int\'eresser aux fonctions satisfaisant aux
relations
\begin{equation}
\label{Invariance}
F(\gamma z g \kappa) = F(g), \quad \gamma \in G(k), \quad z \in Z(\Ad),
\quad g \in G(\Ad), \quad \kappa \in \mathbf{K},
\end{equation}
autrement dit aux fonctions d\'efinies sur la vari\'et\'e modulaire $X$.
\emph{La s\'erie d'Eisenstein normalis\'ee} est d\'efinie pour $g \in
G(\Ad)$ par
$$
\EisCan_{n}(g, s) = \EisCan(g, s) = 
\sum_{\gamma \in P(k) \backslash G(k)} \delta_{P}(\gamma g)^{s}.
$$
On note $\EuScript{S}^{K}(\Ad^{n})$ l'espace des fonctions de
$\EuScript{S}(\Ad^{n})$ invariantes par $\mathbf{K}$. Si $\varphi \in
\EuScript{S}(\Ad^{n})$, on note $\boldsymbol{\phi}_{1}$ la fonction de
$\EuScript{S}(\Ad_{k})$ d\'efinie par
$$\boldsymbol{\phi}_{1}(\lambda) = \varphi(\lambda e_{n}) \qquad (\lambda \in
\Ad_{k}).$$
Si $\phi \in \EuScript{S}(\Ad_{k})$ et si $s \in \CC$, on note
$$
\Tate_{k}(\phi,s) = \Tate_{k}(\phi,s, 1)
$$
l'int\'egrale de Tate \eqref{Weil} de $\phi$ lorsque $\chi = 1$ est le
caract\`ere trivial de $C_{k}$, de telle sorte que
$$
\Tate_{k}(\boldsymbol{\phi}_{1},s) =
\int_{\Ad^{\times}} \varphi(\lambda e_{n}) \, \abs{\lambda}^{s} \,
d^{\times}\!\lambda.
$$

\begin{proposition}
\label{Prop_SE}
Les propri\'et\'es suivantes sont satisfaites :
\begin{enumerate}
\item
\label{EisSt1}
La s\'erie $\EisCan(g, s)$ converge pour $\Re(s) > 1$ et satisfait
aux relations \eqref{Invariance}.
\item
\label{EisSt2}
Si $\varphi \in \EuScript{S}^{K}(\Ad^{n})$, on a
\begin{equation}
\label{pgcd}
\Eis{\varphi}(g,s) =
n \, \Tate_{k}(\boldsymbol{\phi}_{1},ns) \, \EisCan(g, s).
\end{equation}
\item
\label{EisSt3}
La s\'erie $\EisCan(g, s)$ se prolonge en une fonction m\'eromorphe dans
$\CC$.
\end{enumerate}
\end{proposition}

\begin{proof}
Pour $g \in G(\Ad)$ et pour $\varphi \in \EuScript{S}(\Ad^{n})$, on pose
$$
M(\varphi)(g,s) = 
\int_{Z(\Ad)} \varphi(e_{n}.zg) \, \abs{\det zg}^{s} \, dz.
$$
Cette int\'egrale converge si $\Re(s) > 1/n$ et si $\Re(s) > 1$, on a
\begin{equation}
\label{Eisenstein2}
\Eis{\varphi}(g, s) =
\sum_{\gamma \in P(k) \backslash G(k)} M(\varphi)(\gamma g,s).
\end{equation}
Si $g \in G(\Ad)$, on a, en vertu de \eqref{MesZ} :
\begin{eqnarray*}
M(\varphi)(pg,s)
& = &
\int_{Z(\Ad)} \varphi(e_{n}.zpg) \, \abs{\det zpg}^{s} \, dz
\\ & = &
n \, \abs{\det p}^{s} \int_{\Ad^{\times}} \varphi(\lambda e_{n}.pg) \,
\abs{\lambda^{n} \, \det g}^{s} \, d^{\times}\!\lambda
\\ & = &
n \, \abs{\det p}^{s} \int_{\Ad^{\times}} \varphi(\lambda \alpha(p) e_{n}.g) \,
\abs{\lambda^{n} \, \det g}^{s} \, d^{\times}\!\lambda
\\ & = &
n \, \abs{\det p}^{s} \, \abs{\alpha(p)^{-n}}
\int_{\Ad^{\times}} \varphi(\lambda e_{n}.g) \, \abs{\lambda^{n} \,
\det g}^{s} \, d^{\times}\!\lambda
\\ & = &
n \, \delta_{P}(p)^{s} \int_{\Ad^{\times}} \varphi(\lambda e_{n}.g)
\, \abs{\lambda^{n} \, \det g}^{s} \, d^{\times}\!\lambda
\\ & = &
\delta_{P}(p)^{s} \, M(\varphi)(g,s)
\end{eqnarray*}
On a aussi
\begin{eqnarray*}
M(\varphi)(g,s)
& = &
\int_{Z(\Ad)} \varphi(e_{n}.zg) \, \abs{\det zg}^{s} \, dz
\\ & = &
n \int_{\Ad^{\times}} \varphi(\lambda e_{n}.g) \,
\abs{\lambda^{n} \, \det g}^{s} \, d^{\times}\!\lambda
\\ & = &
n \, \abs{\det g}^{s} \, \Tate_{k}(\boldsymbol{\phi}_{g},ns).
\end{eqnarray*}
Si $\varphi \in \EuScript{S}^{K}(\Ad^{n})$, et en \'ecrivant $g = p \kappa \in
G(\Ad)$, avec $p \in P$ et $\kappa \in \mathbf{K}$, on trouve
\begin{equation}
\label{FonctionM}
M(\varphi)(g,s) = n \, \delta_{P}(g)^{s} \, \Tate_{k}(\boldsymbol{\phi}_{1},ns).
\end{equation}
On d\'eduit \eqref{EisSt2} de \eqref{FonctionM} et \eqref{Eisenstein2},
ce qui entra\^ine \eqref{EisSt1} et \eqref{EisSt3} par la m\^eme occasion. 
\end{proof}

On rappelle que la fonction
\begin{equation}
\label{GdLambda}
\RiemXi(s) = \pi^{-s/2} \Gamma(\frac{s}{2}) \zeta(s),
\end{equation}
est m\'eromorphe dans $\CC$, et que $\RiemXi(s) = \RiemXi(1 - s)$. Les seuls
p\^oles de $\RiemXi$ sont en $s = 0$ et $s = 1$. On a $\RiemXi(s) \neq 0$ si $\Re(s)
\geq 1$ et $s \neq 1$. Pour $n = 2$, on a $2 \, \RiemXi(2) = \pi/3$. La
fonction
$$
c_{n}(s) = c(s) = \dfrac{\RiemXi(n(1 - s))}{\RiemXi(ns)}
$$
est m\'eromorphe dans $\CC$, et on a
$$
c(s) \, c(1 - s) = 1, \quad c(\tfrac{1}{2}) = 1, \quad c(0) = 0,
\quad c(\tfrac{1}{n}) = 0 \quad (n \geq 3).
$$
Dans le demi-plan ferm\'e $\Re(s) \geq 1/n$, les seuls p\^oles de la
fonction $c(s)$ sont situ\'es en $s = 1$ et aussi en $s = 1 - (1/n)$ si $n
\geq 3$. On a
$$
\res_{s = 1} c(s) = \dfrac{- 1}{n \, \RiemXi(n)}, \quad
\res_{s = 1 - (1/n)} c(s) = \dfrac{-1}{n \, \RiemXi(n - 1)} \quad (n \geq 3).
$$
On pose
$$
\EuScript{R}_{n} = \set{s \in \CC}{\Re(s) > 0 \ \text{et} \ \zeta(ns) = 0}.
$$
L'ensemble $\EuScript{R}_{n}$ est contenu dans la bande $0 < \Re(s) < 1/n$.
Dans cette bande, l'ensemble des p\^oles de $c(s)$ est \'egal \`a
$\EuScript{R}_{n}$. Dans le demi-plan ferm\'e $\Re(s) \geq 1/2$, l'ensemble
des z\'eros de la fonction $c(s)$ est contenu dans la bande ouverte
$$\set{s \in \CC}{1 - \tfrac{1}{n} < \Re(s) < 1}.$$

\begin{remark}
\label{FonctionTau}
On d\'efinit une fonction $\tau \in \EuScript{S}^{K}(\Ad^{n})$ de la
mani\`ere suivante : on note $\tau_{p}$ la fonction caract\'eristique du
module compact $\ZZ_{p}^{n}$, on pose $\tau_{\infty}(x) = e^{- \pi
\N{x}^{2}}$ si
$x \in \RR^{n}$, et enfin
$$\tau(x) = \tau_{\infty}(x_{\infty}) \prod_{p} \tau_{p}(x_{p}),$$
si $x = (\dots, x_{p}, \dots, x_{\infty}) \in \Ad^{n}$. On v\'erifie que
$\tau \in \EuScript{S}^{K}(\Ad^{n})$, on a $\EuScript{F}\tau(x) = \tau(x)$
et 
\begin{equation}
\label{DeltaTau}
\Tate_{k}(\boldsymbol{\tau}_{1},s) = \RiemXi(s),
\end{equation}
o\`u $\RiemXi(s)$ est d\'efinie en \eqref{GdLambda} \cite[Lem. 8, p.
127]{BNT}. On d\'eduit de \eqref{pgcd} :
\begin{equation}
\label{EisTau}
\Eis{\tau}(g,s) = n \, \RiemXi(ns) \, \EisCan(g, s),
\end{equation}
\end{remark}

\begin{proposition}
\label{Prop_Eisenstein}
La s\'erie d'Eisenstein normalis\'ee v\'erifie les propri\'et\'es suivantes
:
\begin{enumerate}
\item
\label{EisSt4}
Les seuls p\^oles de la fonction $\RiemXi(ns) \, \EisCan(g, s)$ sont les
points $s = 0$ et $s = 1$.
\item
\label{EisSt5}
Dans le demi-plan ferm\'e $\Re(s) \geq 1/n$, la fonction $\EisCan(g, s)$
n'admet qu'un p\^ole ; il est simple et situ\'e en $s = 1$. On a
$$
\res_{s = 1} \EisCan(g, s) = \dfrac{1}{n \, \RiemXi(n)},
\quad
\EisCan(g, 0) = 1, \quad \EisCan(g, \frac{1}{n}) = 0.
$$
\item
\label{EisSt6}
Dans le demi-plan ferm\'e $\Re(s) \geq 1/n$, les seuls p\^oles de la fonction
$c(s)$ sont situ\'es en $s = 1$ et aussi en $s = 1 - (1/n)$ si $n \geq 3$.
On a
$$
\res_{s = 1} c(s) = \dfrac{- 1}{n \, \RiemXi(n)}, \quad
\res_{s = 1 - (1/n)} c(s) = \dfrac{-1}{n \, \RiemXi(n - 1)} \quad (n \geq 3).
$$
\item
\label{EisSt7}
On a l'\'equation fonctionnelle
\begin{equation}
\label{EqFonct2}
\EisCan(g, s) = c(s) \, \EisCan(^t\!g^{-1}, 1 - s).
\end{equation}
Dans la bande $0 < \Re(s) < 1$, le lieu des p\^oles de $\EisCan(g, s)$ est
\'egal \`a $\EuScript{R}_{n}$.
\item
\label{EisSt8}
Si $n = 2$, on a
$$
\EisCan(^t\!g^{-1}, s) = \EisCan(g, s).
$$
\end{enumerate}
\end{proposition}

\begin{proof}
L'assertion \eqref{EisSt4} est une cons\'equence de \eqref{EisTau}. En
\'ecrivant
$$
\EisCan(g, s) = \frac{\Eis{\tau}(g,s)}{n \, \RiemXi(ns)},
$$
on voit que la fonction $\EisCan(g, s)$ n'admet qu'un p\^ole dans le
demi-plan $\Re(s) > 1/n$ ; il est simple et situ\'e en $s = 1$. Puisque
$$
\res_{s = 1} \Eis{\tau}(g,s) = \, \int \tau = 1,
$$
On en d\'eduit le r\'esidu de $\EisCan(g, s)$ en $s = 1$. D'autre part,
$$
\res_{s = 0} \Eis{\tau}(g,s) = - \, \tau(0) = - 1,
$$
et la fonction $\RiemXi(ns)$ admet un p\^ole au point $s = 0$, de r\'esidu
$$
\res_{s = 0} \RiemXi(ns) = 
\res_{s = 0} \Gamma(\frac{ns}{2}) \zeta(ns) = - \dfrac{1}{n},
$$
puisque $\zeta(0) = -1/2$. On en d\'eduit que $\EisCan(g, 0) = 1$. Si
$\Re(s) = 1/n$ et $s \neq 1/n$, la fonction $\EisCan(g, s)$ n'a pas de
p\^ole, puisque $\RiemXi(ns) \neq 0$. Enfin, si $s = 1/n$, la fonction
$\RiemXi(ns)$ a un p\^ole et la fonction $\Eis{\tau}(g,s)$ n'en a pas, ce
qui ach\`eve de d\'emontrer \eqref{EisSt5}. Les seuls p\^oles de la fonction
$\RiemXi(n(1 - s))$ sont simples, situ\'es en $s = 1$ et en $s = 1 - (1/n)$ ;
on a $\RiemXi(ns) \neq 0$ dans l'ensemble des points $s$ tels que $\Re(s)
\geq 1/n$ et $s \neq 1/n$ ; par cons\'equent, les seuls p\^oles de la
fonction 
$$c(s) = \dfrac{\RiemXi(n(1 - s))}{\RiemXi(ns)}$$
dans cet ensemble sont simples, situ\'es en $s = 1$ et en $s = 1 - (1/n)$ ;
de plus
$c(1/n) = 0$. Le calcul de $\res_{s = 1} c(s)$ est facile. On a
$$
\res_{s = 1 - (1/n)} c(s) =
\dfrac{-1}{n} \, \lim_{s = 1} \dfrac{(s - 1) \RiemXi(s)}{\RiemXi(n - s)}.
$$
On a $(s - 1) \RiemXi(s) = 1 + o(1)$, d'o\`u le r\'esidu en $s = 1 - (1/n)$.
Si $n = 2$, on a
$$\RiemXi(2 - s) \sim \dfrac{1}{1 - s},$$
et donc $\res_{s = 1/2} c(s) = 0$ : la fonction $c$ est holomorphe en $s =
1/2$. ce qui implique \eqref{EisSt6}. On d\'eduit de \eqref{pgcd} et de
l'\'equation fonctionnelle
\eqref{EqFct1} :
$$
\RiemXi(n(1 - s)) \, \EisCan(g, 1 - s) =
\RiemXi(ns) \, \EisCan(^t\!g^{-1}, s),
$$
ce qui implique \eqref{EisSt7}. L'assertion \eqref{EisSt8} vient de la
relation $^t\!g^{-1} = w g w^{-1}$, o\`u $w$ est l'\'el\'ement du groupe de
Weyl de $SL(2, \ZZ)$.
\end{proof}

\medskip

\subsection*{Op\'erateurs invariants}~\medskip

La vari\'et\'e riemanienne $\EuScript{P}_{n} = G(\RR)/\Orth_{v}(n)$
s'identifie \`a l'espace des matrices sym\'e\-triques d\'efinies positives
$Y = (y_{ij})$. On note $\mathbf{D} = \mathbf{D}(\EuScript{P}_{n})$ la
$\CC$-alg\`ebre des op\'erateurs diff\'eren\-tiels sur $\EuScript{P}_{n}$
qui sont invariants sous l'action de $G(\RR)$. Posons
$$
\frac{\partial}{\partial Y} = 
\Demi \left( (1 + \delta_{ij}) \ \frac{\partial}{\partial y_{ij}} \right),
$$
o\`u $\delta_{ij}$ est le symbole de Kronecker.
A. Selberg \cite[p. 57]{Selberg} a introduit les \'el\'ements suivants de
$\mathbf{D}$ :
$$
Q_{h} =
\Tr \left( \left( Y \frac{\partial}{\partial Y} \right)^{h} \right) \quad
(1 \leq h \leq n)
$$
de degr\'es respectifs $1,2, \dots, n$ ; voir aussi \cite{Terras}.
L'op\'erateur $Q_{2}$ est l'image de \emph{l'\'el\'e\-ment de Casimir} de
l'alg\`ebre enveloppante de $G$, et c'est aussi le Laplacien de
$\EuScript{P}_{n}$. Il a d\'emontr\'e que l'alg\`ebre $\mathbf{D}$ est
\'egale \`a l'alg\`ebre de polyn\^omes $\CC[Q_{1}, \dots, Q_{n}]$. La
fonction $\delta_{P}(g)^{s}$ est fonction propre simultan\'ee des \'el\'ements
de $\mathbf{D}$ : on a
\begin{equation}
\label{DefCasimir}
D \, \delta_{P}^{s} = \boldsymbol{\gamma}_{D}(s) \, \delta_{P}^{s}, \quad
(D \in \mathbf{D}, s \in \CC),
\end{equation}
avec des polyn\^omes $\boldsymbol{\gamma}_{D}(s)$ convenables.

\begin{proposition}
\label{Prop_Ideal}
Soit $\mathbf{I}(\CC)$ l'image de l'homomorphisme de $\mathbf{D}$ dans
$\CC[s]$ qui envoie $D$ sur $\boldsymbol{\gamma}_{D}(s)$.
\begin{enumerate}
\item
Si $n = 2$, on a $\mathbf{I}(\CC) = \CC[s(1 - s)]$.
\item
Si $n \geq 3$, on a $\mathbf{I}(\CC) = s(1 - )\CC[s]$.
\end{enumerate}
\end{proposition}

Cette proposition r\'esulte du lemme suivant :

\begin{lemma}
\label{Lemme_Gamma}
Posons
$$D_{h} = \frac{2^{h}}{n (n - 1)} \, Q_{h}, \quad
\boldsymbol{\gamma}_{h}(s) = \boldsymbol{\gamma}_{D_{h}}(s).$$
Alors
$$
\boldsymbol{\gamma}_{h}(s) = s \, (1 - s)
\left( \frac{((n - 1)s)^{h - 1} -
(1 - s)^{h - 1}} {(n - 1)s - (1 - s)} \right).
$$
En particulier :
$$
\boldsymbol{\gamma}_{1}(s) = 0, \quad
\boldsymbol{\gamma}_{2}(s) = s \, (1 - s), \quad
\boldsymbol{\gamma}_{3}(s) = s \, (1 - s) \, ((n - 2)s + 1).
$$
\end{lemma}

\begin{proof}
Voir \cite[pp. 44--49]{Terras}.
\end{proof}

Puisque les \'el\'ements de $\mathbf{D}$ sont invariants sous l'action de
$G(\RR)$ \`a gauche, la relation \eqref{DefCasimir} implique :

\begin{proposition}
\label{Prop_FcnPropre}
Si $D \in \mathbf{D}$ et si $s$ n'est pas un p\^ole de $\EisCan(g, s)$,
on a
$$
D \, \EisCan(g, s) = \boldsymbol{\gamma}_{D}(s) \, \EisCan(g, s).
\rlap \qed
$$
\end{proposition}

Rappelons que le radical unipotent $N$ de $P$ est form\'e des matrices
$$
x =
\begin{pmatrix}
\mathbf{1}_{n -1} & u \\
0                 & 1 \\
\end{pmatrix}
$$
o\`u $u$ est un vecteur colonne avec $n - 1$ composantes. Si $F$ est une
fonction continue sur $G(k) \backslash G(\Ad)$, le \emph{terme constant de
$F$ le long de $N$} est
$$F^{0}(g) = \int_{N(k) \backslash N(\Ad)} F(u g, s) \, du.$$
Le terme constant des s\'eries d'Eisenstein est le suivant : si 
$$
p =
\begin{pmatrix} g' & ^{t}\!x \\ 0 & t \\ \end{pmatrix}
\in P(\Ad),
$$
on a
$$
\Eis{\varphi}^{0}(p, s) = n \, \delta_{P}(p)^{s} \,
\Delta(\widehat{\varphi''}, n s) +
\frac{n}{n - 1} \, \delta_{P}(p)^{\frac{1 - s}{n - 1}} \,
\Eis{\varphi'}(g',\frac{n - n s}{n - 1}),
$$
o\`u on a pris $\varphi \in \EuScript{S}(\Ad^{n})$ d\'ecomposable :
$$
\varphi(x', y) =
\varphi'(x') \, \varphi''(y), \quad \varphi'(0) = \varphi''(0) = 1, 
\quad x' \in \Ad^{n - 1}, \ y \in \Ad^{1}.
$$

On en d\'eduit :
\begin{proposition}
\label{Prop_ChowlaSelberg}
Soit $g = p \kappa \in G(\Ad)$ avec $p$ comme ci-dessus. Si $s$ n'est pas un
p\^ole de $\EisCan_{n}(g , s)$, on a
$$
\EisCan_{n}^{0}(g , s) =
\delta_{P}(g)^{s} +
c_{n}(s) \, \delta_{P}(g)^{\frac{1 - s}{n - 1}} \,
\EisCan_{n - 1}(g',\frac{n - n s}{n - 1}),
$$
en convenant que $\EisCan_{1}(g',s) = 1$.
\end{proposition}

\medskip

\subsection*{Int\'egrales p\'eriodiques des s\'eries d'Eisenstein}~

\medskip

Rappelons que $g_{_{K}}$ est la matrice de $G(\Ad)$ telle que
$(g_{_{K}})_{\infty} = q_{_{K}}$ et $(g_{_{K}})_{p} = 1$ pour tout nombre
premier $p$, et que $q_{_{K}} \in G(\RR)$ est une matrice telle que
$^{t}\!q_{_{K}}.q_{_{K}} = p_{_{K}}$, o\`u $p_{_{K}}$ est la matrice
d\'efinie en \eqref{MatSymK}. La fonction $\tau \in \EuScript{S}^{K}
(\Ad^{n})$ a \'et\'e d\'efinie dans la remarque
\ref{FonctionTau}.
 
\begin{lemma}
\label{Lemme_FacGamma1}
Si $\chi \in \mathsf{X}$, on a $\tau(^{t}\!g_{_{K}}.^{t}\!\pi(\xi).e_{1}) =
\Phi_{0}(\xi)$ pour tout $\xi \in \Ad^{\times}_{K}$.
\end{lemma}

\begin{proof}
Les deux fonctions sont d\'ecomposables ; posons $(\tau_{g_{_{K}}})_{v} =
\tau_{v}$ et $(\Phi_{0})_{v} = \Phi_{v}$ pour simplifier. Supposons tout
d'abord que $v$ soit la place infinie. Si $x = (x_{1}, \dots, x_{n}) \in
k^{n}$, et si $\xi = \iota(x) = x_{1}\alpha_{1} + \dots + x_{n}\alpha_{n} \in
K^{\times}$, on v\'erifie que
$${^{t}\!x}.p_{_{K}}.x = F(\xi).$$
On a
$^{t}\!\pi(\xi).e_{1} = \iota^{-1}(\xi) = x$ d'apr\`es \eqref{InvPsi}, et
$$
\N{^{t}\!e_{1}.\pi(\xi)g_{_{K}}}^{2} =
\N{xg_{_{K}}}^{2} = x.p_{_{K}}.{^{t}\!x} = F(\xi).
$$
On en d\'eduit que si $\xi_{\infty} \in K^{\times}_{\infty}$, on a
$$
\tau_{\infty}(^{t}\!g_{_{K}}.^{t}\!\pi(\xi_{\infty}).e_{1}) = e^{-
F(\xi_{\infty})} =
\Phi_{\infty}(\xi_{\infty}).
$$
Soit $v$ une place finie de $k$ ; le caract\`ere $\chi$ est non ramifi\'e, on
a $(g_{_{K}})_{v} = 1$, et
$$
\tau_{v}(^{t}\!\pi_{v}(\xi).e_{1}) = \tau_{v}(\iota^{-1}_{v}(\xi)) \quad
\text{si} \quad \xi \in K^{\times}
$$
puisque $\iota_{v}$ induit un isomorphisme de $\ZZ^{n}_{p}$ sur
$\mathfrak{O}_{v}$. En effet le module
$$
\mathfrak{R}_{v} = \mathfrak{O} \otimes_{\mathfrak{o}} \mathfrak{o}_{v} = 
\alpha_{1}\mathfrak{o}_{v} + \dots + \alpha_{n}\mathfrak{o}_{v}
$$
est un sous-anneau compact de $\mathfrak{O}_{v}$, et il contient
$\mathfrak{O}$. Par le \emph{th\'eor\`eme d'approximation faible},
$\mathfrak{O}$ est dense dans $\prod_{w \mid v} \mathfrak{O}_{w}$, et donc
dans $\mathfrak{O}_{v}$, d'o\`u $\tau_{v} = \Phi_{v}$.
\end{proof}

\begin{lemma}
\label{Lemme_FacGamma2}
Si $\chi \in \mathsf{X}$, on a
$$\Delta'_{K}(\tau_{g_{_{K}}},s,\chi) = \Gamma(s, \chi).$$
\end{lemma}

\begin{proof}
C'est une cons\'equence du lemme \ref{Lemme_FacGamma1} et de
\eqref{DeltaPrime}.
\end{proof}

\begin{proposition}
\label{Prop_FcnQ}
Si $\chi \in \mathsf{X}$ et si $g \in G(\Ad)$, la fonction m\'eromorphe
$$
\Hecke(g, s, \chi) = L(s,\chi)^{-1}
\oint \EisCan(hg, s) \, \chi_{\pi}(h) \, dh
$$
poss\`ede les propri\'et\'es suivantes :
\begin{enumerate}
\item
\label{PFQ1}
Si $h \in T(\Ad)$ et si $\kappa \in \mathbf{K}$, on a
$$
\Hecke(h g \kappa, s, \chi) =
\overline{\chi}_{\pi}(h) \, \Hecke(g, s, \chi).
$$
\item
\label{PFQ2}
La fonction $\RiemXi(n s) \, \Hecke(g, s, \chi)$ est holomorphe dans le
demi-plan $\Re(s) > 0$.
\item
\label{PFQ3}
La fonction $\RiemXi(n s) \, \Hecke(g_{_{K}}, s, \chi)$ n'a pas de
z\'eros dans $\CC$.
\item
\label{PFQ4}
Si $n = 2$, la fonction $\Hecke(g, s, \chi) \, L(s,\chi)$ est invariante par
$s \mapsto 1 - s$.
\end{enumerate}
\end{proposition}

\begin{proof}
Si $\varphi \in \EuScript{S}^{K}(\Ad^{n})$ et si $g \in G(\Ad)$, notons $\Phi_{g}
\in \EuScript{S}(\Ad_{K})$ la fonction d\'efinie par
$$\Phi_{g}(\xi) = \varphi(^{t}\!g.^{t}\!\pi(\xi).e_{1}) \qquad (\xi \in
\Ad_{K}),$$ et $\boldsymbol{\phi}_{1} \in \EuScript{S}(\Ad_{k})$ la fonction
d\'efinie par
$$
\boldsymbol{\phi}_{1}(\lambda) =
\varphi(\lambda e_{1}) \qquad (\lambda \in \Ad_{k}).
$$
La relation \eqref{pgcd} s'\'ecrit 
$$
n \, \Tate_{k}(\boldsymbol{\phi}_{1},ns) \, \EisCan(g, s) =
\Eis{\varphi}(g,s).
$$
La formule de Hecke \eqref{Hecke1} implique
$$
\oint \Eis{\varphi}(hg, s) \, \chi_{\pi}(h) \, dh
= \abs{\det g}_{\Ad_{k}}^{s} \, \Tate_{K}(\varphi_{g},s, \chi),
$$
et la formule de Weil \eqref{Tate} donne :
$$
\Tate_{K}(\Phi_{g},s,\chi) = c_{K}^{-1} \, L(s,\chi) \,
\Tate'_{K}(\Phi_{g},s,\chi).
$$
On en d\'eduit
\begin{eqnarray*}
n \Tate_{k}(\boldsymbol{\phi}_{1},ns)
\oint \EisCan(hg, s) \,
\chi_{\pi}(h) \, dh & = &
\abs{\det g}_{\Ad_{k}}^{s} \, \Tate_{K}(\Phi_{g},s,\chi)
\\ & = &
c_{K}^{-1} \, \abs{\det g}_{\Ad_{k}}^{s} \, L(s,\chi)
\Tate'_{K}(\Phi_{g},s,\chi),
\end{eqnarray*}
ce qui implique
$$
\Hecke(g, s, \chi) = \, \dfrac{\abs{\det g}^{s}}{n \, c_{K}} \,
\dfrac{\Tate'_{K}(\Phi_{g},s,\chi)}{\Tate_{k}(\boldsymbol{\phi}_{1},ns)}.
$$
Puisque $\Tate'_{K}(\Phi_{g},s)$ est holomorphe pour $\Re(s) > 0$, on voit
que $\RiemXi(n s) \, \Hecke(g, s, \chi)$ est holomorphe pour $\Re(s) >
0$ en prenant $\varphi = \tau$, puisque d'apr\`es l'\'equation \eqref{DeltaTau}
$$
\Tate_{k}(\boldsymbol{\tau}_{1},n s) =
\RiemXi(s) = \pi^{- n s/2} \, \Gamma(n s/2) \zeta(n s)
$$
Enfin, d'apr\`es le lemme \ref{Lemme_FacGamma2} et l'\'equation
\eqref{DeltaTau}, on a
\begin{eqnarray*}
\Hecke(g_{_{K}}, s, \chi)
& = &
\dfrac{\abs{\det g_{_{K}}}^{s}}{n \, c_{K}} \,
\dfrac{\Tate'_{K}(\tau_{g_{_{K}}},s,\chi)}{\Tate_{k}(\boldsymbol{\tau}_{1},ns)}
\\ & = & \dfrac{\abs{\disc K}^{s/2}}{n \, c_{K}} \,
\dfrac{\Gamma(s, \chi)}{\pi^{-n s/2} \, \Gamma(n s/2) \zeta(n s)} \, ; \\
\end{eqnarray*}
la fonction $\Gamma(s, \chi)$ n'a pas de z\'eros dans $\CC$, ce qui implique la
derni\`ere assertion.
\end{proof}

Si $\chi \in \mathsf{X}$, on a
$$
\int_{Q} \EisCan(\dot{h}g_{_{K}},s) \, 
\overline{\chi}_{\pi}(\dot{h}) \, d\dot{h}
= \oint F(h g_{_{K}}) \, \overline{\chi}_{\pi}(h) \, dh =
\Hecke(g_{_{K}}, s, \overline{\chi}) \, L(s,\overline{\chi}) \, ;
$$
la formule d'inversion de Fourier \eqref{IntSiegel2} entra\^{\i}ne que si $h
\in S$, on a
$$
\EisCan(hg_{_{K}},s) =
\sum_{\chi \in \mathsf{X}} \ \Hecke(g_{_{K}}, s, \overline{\chi}) \,
L(s,\overline{\chi}) \, \chi_{\pi}(h).
$$
La convergenge est normale car il s'agit du d\'eveloppement en s\'erie de
Fourier d'une fonction $C^{\infty}$ sur une somme de produit de cercles.
On a donc :

\begin{corollary}
\label{IntSiegel3}
si $h \in S$, et si $s \neq 0, 1$, on a
$$
n \Lambda(n s) \, \EisCan(hg_{_{K}},s) =
c_{K}^{-1} \, \abs{\disc K}^{s/2} \,
\sum_{\chi \in \mathsf{X}} \Gamma(s, \chi) \,
L(s,\chi) \, \bar{\chi}_{\pi}(h).
\rlap \qed
$$
\end{corollary}

Cette formule est due \`a Siegel si $K$ est un corps quadratique
r\'eel \cite[p. 89]{Siegel}.

\begin{lemma}
\label{QFcnPropre}
Si $D \in \mathbf{D}$, on a
$$
D\Hecke(g, s, \chi) =
\boldsymbol{\gamma}_{D}(s) \, \Hecke(g, s, \chi).
$$
\end{lemma}

\begin{proof}
Puisque $D$ est invariant \`a gauche, on a, si $g \in G(\Ad)$ et
si $F$ est assez r\'eguli\`ere :
$$
D_{g} \oint F(hg) \, \chi_{\pi}(h) \, dh =
\oint D_{g} [F \, _{\circ} L(h)] (g) \, \chi_{\pi}(h) \, dh =
\oint [DF](hg) \, \chi_{\pi}(h) \, dh.
$$
On en d\'eduit
\begin{eqnarray*}
D\Hecke(g, s, \chi) L(s, \chi) & = &
D_{g} \oint \EisCan(hg, s) \, \chi_{\pi}(h) \, dh
=
\oint [D\mathbf{E}](hg, s) \, \chi_{\pi}(h) \, dh
\\ & = &
\boldsymbol{\gamma}_{D}(s) \Hecke(g, s, \chi) L(s, \chi),
\end{eqnarray*}
d'o\`u le r\'esultat.
\end{proof}

\medskip

\section{Trains d'ondes d'Eisenstein}
\label{sec_TrainsOndes}

\medskip

\subsection*{Trains d'ondes d'Eisenstein}~\medskip

\medskip

Un \textit{train d'ondes d'Eisenstein (fini)} est une combinaison
lin\'eaire de s\'eries d'Eisenstein. Plus pr\'ecis\'ement, c'est une
fonction d\'efinie sur
$G(\Ad)$ qui s'\'ecrit
$$
\Esp{\mu}(g) = \int_{B} \EisCan(g, s) \, d\mu(s),
$$
o\`u $\mu$ appartient \`a l'ensemble $\EuScript{M}_{n}(B)$ des mesures dans
la bande ouverte
$$
B = \set{s \in \CC}{0 < \Re(s) < 1},
$$
et \`a support fini, disjoint de l'ensemble 
$$
\EuScript{R}_{n} = \set{s \in \CC}{\Re(s) > 0 \ \text{et} \ \RiemXi(ns) = 0},
$$
qui est l'ensemble des p\^oles de $\EisCan(g, s)$ dans $B$. On parlera
aussi de \emph{train d'ondes} pour simplifier. Si
$$
\mu = \sum_{s} \ a_{s} \delta_{(s)},
\quad \quad a_{s} \in \CC, \quad s \in B,
$$
on a
$$
\Esp{\mu}(g) = \sum_{s} \ a_{s} \, \EisCan(g, s).
$$
Les fonctions
$\Esp{\mu}$ appartiennent \`a l'espace $C(X)$ des fonctions continues
d\'efinies sur la \emph{vari\'et\'e modulaire}
$$X = G(k)Z(\Ad) \backslash G(\Ad) / \mathbf{K}.$$
On note $\Toe(X) = \Im \mathsf{W}$ l'espace des trains d'ondes
d'Eisenstein. On a
\begin{equation}
\label{EqFcnTO}
\Esp{\mu}(g)
= \int_{B} c(s) \, \EisCan(^t\!g^{-1}, 1 - s) \, d\mu(s)
= \int_{B} c(1 - s) \, \EisCan(^t\!g^{-1}, s) \, d\mu(1 - s)
\end{equation}
par l'\'equation fonctionnelle \eqref{EqFonct2} :
$$
\EisCan(g, s) = c(s) \, \EisCan(^t\!g^{-1}, 1 - s).
$$

On va maintenant d\'eterminer le noyau de l'application $\mathsf{W}$.

Supposons $n = 2$. On note $\EuScript{M}^{\pm}_{2}(B)$ le sous-espace de
$\EuScript{M}_{2}(B)$ form\'e des mesures telles que
$$\mu(1 - s) = \pm \, c(s) \, \mu(s).$$
Si $\mu \in \EuScript{M}_{2}(B)$, on pose
$$
\mu^{\pm}(s) = \dfrac{1}{2} \, (\mu(s) \pm c(1 - s) \, \mu(1 - s) )
\in \EuScript{M}^{\pm}_{2}(B).
$$
Puisque $\mu = \mu^{+} + \mu^{-}$, on a
$$
\EuScript{M}_{2}(B) = \EuScript{M}^{+}_{2}(B) \oplus \EuScript{M}^{-}_{2}(B),
$$
et l'\'equation fonctionnelle \eqref{EqFcnTO} implique
$\Esp{\mu} = \Esp{\mu^{+}}$.

\begin{proposition}
\label{Prop_Indep}
Si $n \geq 3$, l'application $\mathsf{W} : \EuScript{M}_{n}(B)
\longrightarrow C(X)$ est injective. Si $n = 2$, on a
$\ker \mathsf{W} = \EuScript{M}^{-}_{2}(B)$.
\end{proposition}

Autrement dit, les s\'eries d'Eisenstein $\EisCan_{n}(g, s)$ sont
lin\'eairement ind\'ependantes, mis \`a part la relation $\EisCan_{2}(g,
s) = c(s) \, \EisCan_{2}(g, 1 - s)$ lorsque $n = 2$.

Posons $\EuScript{M}^{+}_{n}(B) = \EuScript{M}_{n}(B)$ pour $n \geq 3$. Si
$F \in C(X)$ est un train d'ondes d'Eisenstein, il existe une unique mesure
$\mu \in \EuScript{M}^{+}_{n}(B)$ telle que $F = \Esp{\mu}$ ; on dit que le
support de $\mu$ est le \emph{spectre} de $F$, not\'e $\Spec F$.

\begin{proof}
Supposons
$$
\mu = \sum_{i = 1}^{n} \, a_{i} \, \delta_{(s_{i})},
\quad a_{i} \in \CC, \quad s_{i} \in B, \quad
s_{i} \neq s_{j} \ \text{si} \ i \neq j.
$$
Soit $g \in G(\Ad)$, et posons $\nu_{g}(s) = \EisCan(g, s) \, \mu(s) \in
\EuScript{M}_{n}(B)$. Si $D \in \mathbf{D}$, on a
$$
D[W(\mu)](g) =
\sum_{i = 1}^{n} \, a_{i} \, \boldsymbol{\gamma}_{D}(s_{i})
\EisCan(g,s_{i}) = \int_{B} \boldsymbol{\gamma}_{D}(s) d\nu_{g}(s).
$$
Si $W(\mu) = 0$, on a donc
$$\int_{B} \boldsymbol{\gamma}(s) d\nu_{g}(s) = 0$$
quel que soit $\boldsymbol{\gamma} \in \mathbf{I}(\CC)$. Supposons $n
\geq 3$. Pour $1 \leq j \leq n$, Le polyn\^ome
\begin{equation}
\label{BasePol}
\boldsymbol{\gamma}_{j}(s) =
\frac{s(1 - s)}{s_{j}(1 - s_{j})} \ \prod_{i \neq j}
\frac{s - s_{j}}{s_{i} - s_{j}}
\end{equation}
appartient \`a $\mathbf{I}(\CC) = s(1 - s)\CC[s]$ et
$$
\int_{B} \boldsymbol{\gamma}_{j}(s) d\nu_{g}(s) =
\sum_{i = 1}^{n} \, a_{i} \, \boldsymbol{\gamma}_{j}(s_{i})
\EisCan(g,s_{i}) = a_{i} \EisCan(g,s_{i}) = 0,
$$
quel que soit $g \in G(\Ad)$ ce qui implique $a_{i} = 0$ puisque la
fonction $\EisCan(g,s_{i})$ n'est pas identiquement nulle. On a donc
$\mu = 0$ et l'application $\mathsf{W}$ est injective. Supposons $n =
2$, et soit $\mu \in \EuScript{M}^{+}_{2}(B)$. La proposition
\ref{Prop_FcnQ}\eqref{PFQ4} implique
$$
\nu_{g}(s) = \EisCan(g, s) \, \mu(s) = c(s) \,
\EisCan(g, 1 - s) \, \mu(s) = \EisCan(g, 1 - s) \, \mu(1 - s) =
\nu_{g}(1 - s)
$$
La mesure $\nu_{g}$ est invariante par $s \mapsto 1 - s$, et elle
est nulle sur $\mathbf{I}(\CC) = \CC[s(1 - s)]$ ; on a encore $\nu_{g} = 0$
dans ce cas. On conclut comme pr\'ec\'edemment.
\end{proof}

\medskip

\subsection*{Formes toro\"\i dales}~

\medskip

Les fonctions $F$ d\'efinies sur $G(k)Z(\Ad) \backslash G(\Ad)$ telles que
$$\oint F(h g) \, dh = 0 \quad \text{pour} \quad g \in G(\Ad)$$
ont \'et\'e introduites par Zagier \cite{Zagier}. 
Plus g\'en\'eralement, si $\chi \in \mathsf{X}$, on dit que $F \in C(X)$
est une \emph{forme toro\"\i dale en $\chi$} si
\begin{equation}
\label{FormeTorique}
\Pi_{\chi}(F)(g) = \oint F(h g) \, \chi_{\pi}(h) \, dh = 0
\quad \text{pour} \ g \in G(\Ad),
\end{equation}
On note $\Tor$ l'espace des formes toro\"\i dales en $\chi$.

\medskip

\subsection*{Trains d'ondes toro\"\i daux}~

\medskip

\begin{theorem}
\label{Thm_CritTrainTor}
Soit $F \in \Toe(X)$ un train d'ondes d'Eisenstein et $\chi \in
\mathsf{X}$. Les conditions suivantes sont \'equivalentes :
\begin{enumerate}
\item
\label{CTT1}
Le train d'ondes $F$ est toro\"\i dal en $\chi$, \idest $\Pi_{\chi}(F) = 0$ ;
\item
\label{CTT2}
On a $\Spec F \subset \EuScript{Z}_{\chi}$, o\`u
$$
\EuScript{Z}_{\chi} = \set{s \in B}{L(s, \chi) = 0 \, }.
$$
\end{enumerate}
\end{theorem}

La proposition \ref{Prop_Indep} et le th\'eor\`eme \ref{Thm_CritTrainTor}
impliquent que l'application $\mathsf{W}$ induit un isomorphisme de l'espace
$$
\set{\mu \in \EuScript{M}^{+}_{n}(B)}
{\Supp \mu \subset \EuScript{Z}_{\chi}}
$$
sur l'espace $\Toe(X) \cap \Tor$ des trains d'ondes qui sont toro\"\i daux en
$\chi$.

En particulier, si $s \in B - \EuScript{R}_{n}$, \emph{pour que
$\EisCan(g, s)$ soit toro\"\i dale en $\chi$, il faut et il suffit que $s
\in \EuScript{Z}_{\chi}$}.

\begin{proof}
Ecrivons $F = \Esp{\mu}$ avec $\mu \in \EuScript{M}_{n}(B)$ et posons
$\Pi(g) = \Pi_{\chi}(F)(g)$. La formule de Hecke implique
$$
\Pi(g) = \oint \Esp{\mu}(hg) \, \chi_{\pi}(h) \, dh =
\int_{B} \Hecke(g, s, \chi) \, L(s, \chi) \, d\mu(s).
$$
Rappelons que pour toute mesure $\mu \in \EuScript{M}_{n}(B)$, on a $\Supp \mu
\subset \EuScript{Z}_{\chi}$ si et seulement si la mesure $L(s, \chi)
\mu(s) = 0$ est nulle.

\eqref{CTT2} $\Rightarrow$ \eqref{CTT1} :
Si $L(s, \chi) \, \mu(s) = 0$, alors $\Esp{\mu}$ est
toro\"\i dale en $\chi$ gr\^ace \`a l'\'egalit\'e ci-dessus.

\eqref{CTT1} $\Rightarrow$ \eqref{CTT2} :
Soit $D \in \mathbf{D}$. Gr\^ace au lemme \ref{QFcnPropre}, on a
\begin{eqnarray*}
D\Pi(g) & = &
\int_{B} D\Hecke(g, s, \chi) \, L(s, \chi) \, d\mu(s)
\\ & = &
\int_{B}
\Hecke(g, s, \chi) \, L(s, \chi) \, \boldsymbol{\gamma}_{D}(s) \, d\mu(s).
\end{eqnarray*}
si $\Esp{\mu}$ est toro\"\i dale en $\chi$, alors
$$
\int_{B} \boldsymbol{\gamma}(s) \, d\nu(s) = 0 \quad \text{o\`u} \
\nu(s) = \Hecke(g, s, \chi) \, L(s, \chi) \, \mu(s),
$$
quelle que soit $\boldsymbol{\gamma} \in \mathbf{I}(\CC)$.
i) Supposons $n \geq 3$ et posons
$$\mu = \sum_{i = 1}^{n} \, a_{i} \, \delta_{(s_{i})},$$
avec $s_{i} \neq s_{j}$ pour $i \neq j$, de telle sorte que
$$
\nu = \sum_{i = 1}^{n} \, \nu_{i} \, \delta_{(s_{i})}, \quad
\nu(s_{i}) = a_{i} \, \Hecke(g, s_{i}, \chi) \, L(s_{i}, \chi).
$$
Il vient
$$
\int_{B} \boldsymbol{\gamma}(s) \, d\nu(s) =
\sum_{i = 1}^{n} \, \nu_{i} \, \boldsymbol{\gamma}(s_{i}) = 0
\quad (\boldsymbol{\gamma} \in \mathbf{I}(\CC)).
$$
Supposons $n \geq 3$ et reprenons le polyn\^ome $\boldsymbol{\gamma}_{j}(s)
\in \mathbf{I}(\CC)$ d\'efini par \eqref{BasePol} ; on a
$$
\int_{B} \boldsymbol{\gamma}_{j}(s) d\nu(s) =
\sum_{i = 1}^{n} \, \nu_{i} \, \boldsymbol{\gamma}_{j}(s_{i}) = \nu_{i} = 0,
$$
et $\nu = 0$. Or la fonction $\RiemXi(ns) \Hecke(g_{_{K}}, s, \chi)$ ne
s'annule pas dans $B$ d'apr\`es la proposition \ref{Prop_FcnQ} ; il
s'ensuit que la mesure $L(s, \chi) \, \mu(s)$ est nulle.
ii) Supposons $n = 2$, et soit $\mu \in \EuScript{M}^{+}_{2}(B)$. La
proposition \ref{Prop_FcnQ}\eqref{PFQ4} implique
\begin{eqnarray*}
\nu(s) & = & \Hecke(g, s, \chi) \, L(s,\chi) \, \mu(s) =
\\ & = & 
\Hecke(g, 1 - s, \chi) \, L(1 - s,\chi) c(s) \, \mu(s)
\\ & = & 
\Hecke(g, 1 - s, \chi) \, L(1 - s,\chi) \mu(1 - s) = \nu(1 - s).
\end{eqnarray*}
Ainsi, la mesure $\nu$ est invariante par $s \mapsto 1 - s$, et elle
est nulle sur $\mathbf{I}(\CC) = \CC[s(1 - s)]$ ; on a encore $\nu = 0$
dans ce cas. On conclut comme pr\'ec\'edemment.
\end{proof}

\medskip

\subsection*{Trains d'ondes principaux}~

\medskip

En th\'eorie du signal, un train d'ondes est repr\'esent\'e par une
somme finie d'exponentielles imaginaires
$$\sum_{t} \, a_{t} \, e^{it\omega}.$$
Les exponentielles imaginaires sont les caract\`eres du groupe additif.
Les caract\`eres du groupe multiplicatif $\RRm$ sont les fonctions
puissances $\psi_{t}(x) = x^{it}$, o\`u $t$ est r\'eel. On note
$\Toe^{1}(\RRm)$ l'espace des sommes de caract\`eres
$$
\psi(x) = \sum_{t} a_{t}(\psi) \, e^{it\omega} = \sum_{t} a_{t}(\psi) \,
x^{it},
$$
o\`u l'ensemble $\Spec \psi = \set{t \in \RR}{a_{t}(\psi) \neq 0}$ est
\emph{fini}, de telle sorte que
$$\psi(x) = \dint x^{it} \, d\widehat{\psi}(t),$$
o\`u $\widehat{\psi}$ est une mesure \`a support fini sur $\RR$ :
$$
\widehat{\psi} = \sum_{t \in \Spec \psi} a_{t}(\psi) \, \delta_{(t)}.
$$
Si les s\'eries d'Eisenstein $E(g, \Demi + it)$ sont l'analogue pour
l'espace $X$ des caract\`eres $\psi_{t}(x)$ pour $\RRm$, les trains d'ondes
d'Eisenstein principaux sont l'analogue des sommes de puissances. On dit
qu'un train d'ondes $F$ est \emph{principal} si son spectre est contenu
dans la \emph{droite critique}
$$D = \set{s \in \CC}{\Re(s) = \Demi}.$$
Il s'\'ecrit donc sous la forme
$$F(g) = \dint E(g, \Demi + it) \, d\widehat{\psi}(t),$$
o\`u $\widehat{\psi}$ est comme ci-dessus. On note $\Toe^{1}(X)$ l'espace
des trains d'ondes principaux. 

\begin{remark}
\label{Rem_SuppAxe}
On d\'eduit du th\'eor\`eme \ref{Thm_CritTrainTor} que si $F\in \Toe(X)$,
les conditions suivantes sont \'equivalentes :
\begin{enumerate}
\item
\label{PSA1}
On a $F \in \Toe^{1}(X) \cap \Tor$.
\item
\label{PSA2}
On a $\Spec F \subset \EuScript{X}_{\chi}$, o\`u
$$
\EuScript{X}_{\chi} = D \cap \EuScript{Z}_{\chi} = \set{s \in \CC}
{L(s, \chi) = 0 \ \text{et} \ \Re(s) = \Demi}.
\rlap \qed
$$
\end{enumerate}
\end{remark}

\medskip

\subsection*{Une condition \'equivalente \`a l'hypoth\`ese de Riemann}~

\medskip

\begin{corollary}
\label{Cor_EquivRiemann}
Les conditions suivantes sont \'equivalentes :
\begin{enumerate}
\item
\label{EQR1}
Toutes les racines de $L(s, \chi)$ dans la bande $B$ sont situ\'ees sur la
droite critique $D$.
\item
\label{EQR2}
Tout train d'ondes d'Eisenstein toro\"\i dal en $\chi$ est principal.
\item
\label{EQR3}
Si $s \in B$ et si $\EisCan(g, s)$ est toro\"\i dale en $\chi$, on a
$$
\EisCan(pg, s) = O(\delta_{P}(p)^{1/2})
$$
pour tout $p \in P(\Ad)$, uniform\'ement lorsque $g$ parcourt un
compact de $G(\Ad)$.

\hfill \qed
\end{enumerate}
\end{corollary}

\begin{proof}
Le th\'eor\`eme \ref{Thm_CritTrainTor} implique imm\'ediatement
l'\'equivalence des conditions \eqref{EQR1} et \eqref{EQR2}. Puisque
$$F(pg) \sim F^{0}(pg),$$
La proposition \ref{Prop_ChowlaSelberg} implique l'\'equivalence des
conditions \eqref{EQR1} et \eqref{EQR3}.
\end{proof}

\medskip

\section{Relations de Maass-Selberg}
\label{sec_MaassSelberg}

\subsection*{L'op\'erateur de troncature}~

\medskip

On va utiliser l'op\'erateur d'Arthur pour estimer l'int\'egrale des
s\'eries d'Eisenstein. 

\begin{proposition}
\label{Prop_OpArthur}
pour $m > 0$, il existe un op\'erateur $\Lambda^{m}$ de $C(G(k) \backslash
G(\Ad))$ jouissant des propri\'et\'es suivantes.
\begin{enumerate}
\item
\label{ARTIntro}
Si $F$ est une fonction continue \emph{\`a
croissance lente} sur $G(k) \backslash G(\Ad)$ et si $m > 0$, la fonction
$\Lambda^{m}F$ est \emph{\`a d\'ecroissance rapide}, et $\Lambda^{m}F$ tend
vers $F$ uniform\'ement sur tout compact lorsque $m$ tend vers l'infini.
\item
\label{ARTdim2}
Supposons $n = 2$. Soient $s_{1}$ et $s_{2}$ deux points du plan complexe
qui ne sont ni des p\^oles de $\EisCan(g, s)$, ni des p\^oles de $c(s)$.
Si $m$ tend vers l'infini, on a :
\begin{equation}
\label{MSDim2}
\int_{G(k) Z(\Ad) \backslash G(\Ad)}
\Lambda^{m}\EisCan(g, s_{1}) \,
\Lambda^{m}\EisCan(g, s_{2}) \, dg \sim v \, \varpi_{2}^{m}(s_{1}, s_{2}),
\end{equation}
o\`u $v > 0$ et o\`u
\begin{eqnarray*}
\label{DefFcnVieuxPi2}
\varpi_{2}(s_{1}, s_{2}) & = &
\frac{m^{s_{1} - s_{2}} c(s_{2}) -
m^{s_{2} - s_{1}} c(s_{1})}{s_{1} - s_{2}} \\
& + & \frac{m^{s_{1} + s_{2} - 1} - m^{1 - s_{2} - s_{1}} c(s_{1}) c(s_{2})}
{s_{1} + s_{2} - 1}.
\end{eqnarray*}
\item
\label{ARTdim3}
Supposons $n \geq 3$. Si $t_{1}$ et $t_{2}$ sont r\'eels, et si $m$
tend vers l'infini, on a
\begin{equation}
\label{MSDim3Reel}
\int_{G(k)Z(\Ad) \backslash G(\Ad)}
\Lambda^{m} \EisCan(g, \Demi + i t_{1}) \,
\Lambda^{m} \EisCan(g, \Demi + i t_{2}) \, dg
\sim v \, M_{1}(t_{1}, t_{2}),
\end{equation}
o\`u $v$ est une constante $> 0$, et o\`u
$$
M_{1}(t_{1}, t_{2}) =
\frac{m^{i(t_{1} + t_{2})} - m^{i(t_{2} + t_{1})}
c(\tfrac{1}{2} + i t_{1})c(\tfrac{1}{2} + i t_{2})}{i(t_{1} + t_{2})},
\quad (t_{2} \neq - t_{1}),
$$
et
$$
M_{1}(t, - t) = 2 \log m - \, \frac{c'}{c}(\tfrac{1}{2} + i t).
$$
\end{enumerate}
\end{proposition}

Les relations \eqref{MSDim2} sont les \emph{relations de Maass-Selberg}, et
les relations \eqref{MSDim3Reel} sont les \emph{relations de Maass-Selberg
g\'en\'eralis\'ees}.

\begin{proof}
Rappelons les propri\'et\'es de \emph{l'op\'erateur de troncature
$\Lambda^{T}$ d'Arthur} d\'efini dans \cite[p. 270]{Arthur1}, \cite[p.
89]{Arthur2}, \cite[p. 40]{Arthur3}, o\`u $T$ est un
\'el\'ement \emph{convenablement r\'egulier} de l'alg\`ebre de Lie du tore
maximal standard $A_{0}$ de $G$. Si $F$ est une fonction continue \emph{\`a
croissance lente} sur $G(k) \backslash G(\Ad)$ et si $T > 0$, la fonction
$\Lambda^{T}F$ est \emph{\`a d\'ecroissance rapide}, et $\Lambda^{T}F$ tend
vers $F$ uniform\'ement sur tout compact lorsque $T$ tend vers l'infini. Si
on note $H_{\rho}$ l'\'el\'ement tel que $<\alpha, H_{\rho}>{} = 1$ pour
toute racine simple $\alpha$ de $G$ relative \`a $A_{0}$, et si $m$ est un
nombre r\'eel assez grand, l'\'el\'ement $T(m) = (\log m)H_{\rho}$ est
convenablement r\'egulier (voir la d\'efinition en \cite[Eq. (9.2), p.
69]{Arthur3}. L'op\'erateur $\Lambda^{m} = \Lambda^{T(m)}$ satisfait la
condition \eqref{ARTIntro}. Si $n = 2$, voir \cite[Prop. 7.13, p.
401]{Knapp2} pour la d\'emonstration de \eqref{ARTdim2}. Si $n \geq 3$, on
d\'eduit \eqref{ARTdim3} des travaux d'Arthur \cite[p. 271]{Arthur1},
\cite[Lem. 4.2, p. 119]{Arthur2}, \cite[Cor. 9.2, p. 70]{Arthur3}.
\end{proof}

Signalons en passant que si $n = 2$, on d\'eduit de \eqref{MSDim2} en
passant
\`a la limite :
$$
\int_{G(k)Z(\Ad) \backslash G(\Ad)}
\abs{\Lambda^{m} 1}^{2} \, dg = v (2 \Lambda(2) - \dfrac{1}{m}) \sim
\vol G(k)Z(\Ad) \backslash G(\Ad),
$$
d'o\`u
$$
v = \frac{\vol G(k)Z(\Ad) \backslash G(\Ad)}{2 \Lambda(2)} =
\frac{3}{\pi} \, \vol G(k)Z(\Ad) \backslash G(\Ad).
$$
La relation \eqref{MSDim2} et un autre passage \`a la limite impliquent :

\begin{lemma}
\label{Lemme_MScomplexe}
Supposons $n = 2$. Si $s \in \CC$ et $\Re(s) > 1/2$, on a
$$
\int_{G(k) Z(\Ad) \backslash G(\Ad)} \abs{\Lambda^{m}E(g, s)}^{2} \, dg =
v \, \dfrac{m^{2 \sigma - 1}}{2 \sigma - 1} + O(1).
$$
\end{lemma}

\medskip

\subsection*{Fonctions presque p\'eriodiques}~

\medskip

Si $\psi \in C(\RRm)$, on d\'efinit une semi-norme $\N{\psi}_{2}$ en posant
$$
\N{\psi}_{2}^{2} =
\lim_{m \rightarrow \infty} \dfrac{1}{2 \log m} \int_{1/m}^{m}
\abs{\psi(x)}^{2} \, d^{\times}\!x.
$$
L'espace de Hilbert $B^{2}(\RRm)$ des \emph{fonctions presque
p\'eriodiques} au sens de Besicovitch, de carr\'e int\'egrable en moyenne
sur $\RRm$, est le compl\'et\'e de l'espace $\Toe^{1}(\RRm)$ des sommes de
caract\`eres pour cette semi-norme. L'espace $B^{2}(\RRm)$ s'identifie \`a
l'espace $\ell^{2}(\RR_{d})$, o\`u $\RR_{d}$ est le sous-groupe discret
sous-jacent \`a la droite r\'eelle, ou encore \`a l'espace des s\'eries
formelles
$$
\psi(x) \sim \sum_{t} a_{t}(\psi) \, x^{i t},
\quad \text{avec} \quad
\sum_{t} \, \abs{a_{t}(\psi)}^{2} < + \infty.
$$
De la m\^eme mani\`ere, si $F \in C(X)$, on d\'efinit une semi-norme
$\N{F}_{2}$, \`a valeurs dans $[0, + \infty]$, en posant
\begin{equation}
\label{NormeArthur}
\N{F}_{2}^{2} = \lim_{m \rightarrow \infty}
\dfrac{1}{4 \, v \log m} \int_{G(k)Z(\Ad) \backslash G(\Ad)}
\abs{\Lambda^{m}F(g)}^{2} \, dg,
\end{equation}
et on dit que $F$ est \emph{de carr\'e int\'egrable en moyenne} si
$\N{F}_{2} < + \infty$. L'espace $\EuScript{A}^{2}(X)$ des fonctions de
carr\'e int\'egrable en moyenne jouit d'une structure d'espace
pr\'ehilbertien : si $F_{1}$ et $F_{2}$ appartiennent \`a
$\EuScript{A}^{2}(X)$, la limite
$$
(F_{1}, F_{2})_{2} =
\lim_{m \rightarrow \infty} \dfrac{1}{4 \, v \log m}
\int_{G(k)Z(\Ad) \backslash G(\Ad)}
\Lambda^{m}F_{1}(g) \, \overline{\Lambda^{m}F_{2}(g)} \, dg
$$
existe et d\'efinit une forme hermitienne sur $\EuScript{A}^{2}(X)$. On note
$A^{2}(X)$ l'espace de Hilbert s\'epar\'e compl\'et\'e de
$\EuScript{A}^{2}(X)$ relativement \`a la semi-norme $\N{F}_{2}$. Or on a
$\Toe^{1}(X) \subset A^{2}(X)$, par la proposition \ref{Prop_SuppPp}
ci-dessous ; ceci permet de d\'efinir par analogie l'espace de
Hilbert $B^{2}(X)$ des \emph{trains d'ondes presque p\'eriodiques} comme
\'etant l'adh\'erence de $\Toe^{1}(X)$ dans $A^{2}(X)$ ; c'est l'analogue
pour la vari\'et\'e modulaire $X$ de l'espace $B^{2}(\RRm)$. Puisque
$$
M_{1}(t, - t) = 2 \log m - \, \frac{c'}{c}(\tfrac{1}{2} + i t),
$$
On d\'eduit de la proposition \ref{Prop_OpArthur} :

\begin{lemma}
\label{Lemme_NormEisDim3}
Supposons $n \geq 3$. Soient $t_{1}$ et $t_{2}$ dans $\RR$.
\begin{enumerate}
\item
\label{MSobliqueGen}
Si $t_{2} \neq t_{1}$, on a
$
(\EisCan(. \, , \tfrac{1}{2} + i t_{1}),
\EisCan(. \, , \tfrac{1}{2} + i t_{2}))_{2} = 0.
$
\item
\label{MSnormeGen}
On a
$\N{\EisCan(. \, , \tfrac{1}{2} + i t)}_{2}^{2} = \Demi$.
\hfill \qed
\end{enumerate}
\end{lemma}

On a aussi :

\begin{lemma}
\label{Lemme_NormEisDim2}
Supposons $n = 2$. Soient $t_{1}$ et $t_{2}$ dans $\RR$.
\begin{enumerate}
\item
\label{MSoblique}
Si $t_{2} \neq \pm \, t_{1}$, on a
$
(\EisCan(. \, , \tfrac{1}{2} + i t_{1}),
\EisCan(. \, , \tfrac{1}{2} - i t_{2}))_{2} = 0.
$
\item
\label{MSnorme}
Si $t \neq 0$, on a
$\N{\EisCan(. \, , \tfrac{1}{2} + i t)}_{2}^{2} = \Demi$.
\item
\label{MSUnDemi}
On a
$\N{\EisCan(. \, , \Demi)}_{2}^{2} = 1$.
\item
\label{MSoppose}
Si $t \neq 0$, on a
$
(\EisCan(. \, , \tfrac{1}{2} + i t),
\EisCan(. \, , \tfrac{1}{2} - i t))_{2} = \Demi \, c(\tfrac{1}{2} + i t).
$
\hfill \qed
\end{enumerate}
\end{lemma}

\begin{proposition}
\label{Prop_SuppPp}
Si $F$ est un train d'ondes principal non nul, on a
$$0 < \N{F}_{2} < + \infty,$$
autrement dit, on a $\Toe^{1}(X) \subset A^{2}(X)$ ; si $n = 2$, on a
$\Toe^{1}(X) = \Toe(X) \cap A^{2}(X)$.
\end{proposition}

\begin{proof}
Supposons $F = \Esp{\mu}$, avec
$$
\mu = \sum_{s} \ a_{s} \delta_{(s)}, \quad a_{s} \in \CC, \quad \Re(s) = \Demi.
$$
Les lemmes \ref{Lemme_NormEisDim3} et \ref{Lemme_NormEisDim2} montrent que
$\Esp{\mu} \in A^{2}(X)$. R\'eciproquement, supposons $n
= 2$. Le lemme \ref{Lemme_MScomplexe} montre que
$$
\int_{G(k) Z(\Ad) \backslash G(\Ad)} \abs{\Esp{\mu}(g)}^{2} \, dg \sim
C \, \dfrac{m^{2 \sigma(\mu) - 1}}{2 \sigma(\mu) - 1}.
$$
Si $\Supp \mu \not \subset D$, on n'a donc pas $\Esp{\mu} \in
\EuScript{A}^{2}$.
\end{proof}

\medskip

\subsection*{Une autre condition \'equivalente \`a l'hypoth\`ese de Riemann
($n = 2$)}~

\medskip

On d\'eduit du th\'eor\`eme \ref{Cor_EquivRiemann} et de la proposition
\ref{Prop_SuppPp} :

\begin{corollary}
\label{Cor_EquivRiemannPp}
Consid\'erons les conditions suivantes :
\begin{enumerate}
\item
\label{EPP1}
Toutes les racines de $L(s, \chi)$ dans la bande $B$ sont situ\'ees sur la
droite critique.
\item
\label{EPP2}
Tout train d'ondes toro\"\i dal en $\chi$ est de carr\'e int\'egrable en
moyenne.
\end{enumerate}
Alors \eqref{EPP1} $\Rightarrow$
\eqref{EPP2}. Si $n = 2$, les conditions
\eqref{EPP1} et \eqref{EPP2} sont \'equivalentes.
\hfill \qed
\end{corollary}

La condition \eqref{EPP2} du th\'eor\`eme \ref{Cor_EquivRiemann} signifie
que la semi-norme $\N{F}_{2}$ est finie sur l'espace $\Toe(X) \cap \Tor$ ;
sous cette condition, l'espace $\Toe(X) \cap \Tor$, muni de cette
semi-norme, est un espace pr\'ehilbertien.

\medskip

\section{Un espace de P\'olya-Hilbert modulaire}
\label{sec_PolyaHilbert}

\medskip

\subsection*{L'espace $T^{2}(X)$}~

\medskip

L'espace de Hilbert $T^{2}_{\chi}(X)$ des \emph{trains d'ondes toro\"\i daux en
$\chi$ et presque p\'eri\-odiques} est l'adh\'erence du sous-espace 
$\Toe^{1}(X) \cap \Tor$ de $B^{2}(X)$. On \'ecrit $T^{2}_{\chi}(X) =
T^{2}(X)$ pour simplifier.

On note $V(D_{\chi})$ le compl\'et\'e de $\Toe^{1}(X) \cap \Tor$ pour la
norme $\N{D F}_{2}$, o\`u $D = D_{2}$ est l'op\'erateur de Casimir de $G$
d\'efini dans la section \ref{sec_Eisenstein}. Puisque $\N{4 D F}_{2} \geq
\N{F}_{2}$, l'espace $V(D_{\chi})$ s'identifie \`a un sous-espace dense de
$T^{2}(X)$. L'op\'erateur non born\'e de $T^{2}(X)$ d\'efini par $D$ et de
domaine $V(D_{\chi})$ sera not\'e $D_{\chi}$.

\begin{theorem}
\label{Thm_PoHil}
L'op\'erateur $D_{\chi}$ est un op\'erateur auto-adjoint de $T^{2}(X)$, et
son spectre est discret : on a
$$
\Spec D_{\chi} =
\set{\lambda}
{\lambda = \frac{1}{4} + \gamma^{2}, \quad
\gamma \in \RR, \quad L(\Demi + i\gamma, \chi) = 0 \,}.
$$
Si $L(\Demi, \chi) = 0$, la valeur propre $\lambda = 1/4$ est simple. Si $n
\geq 3$, toute valeur propre $\lambda \neq 1/4$ est double, les fonctions
propres correspondantes \'etant la s\'erie d'Eisenstein $\EisCan(g,
\Demi + i\gamma)$ et sa conjugu\'ee. Si $n = 2$, toute valeur propre est
simple.
\end{theorem}

Autrement dit, le couple $(T^{2}(X),D_{\chi})$ est un espace de P\'olya-Hilbert
dans le sens expliqu\'e dans l'introduction. Nous d\'ecrivons d'abord la
structure de $\Toe^{1}(X)$.

\begin{proposition}
\label{Prop_MaSe}
Les propri\'et\'es suivantes sont satisfaites.
\begin{enumerate}
\item
\label{MSHilb2}
Supposons $n \geq 3$. Si on \'ecrit les \'el\'ements de $\Toe^{1}(X)$ sous
la forme
\begin{equation}
\label{EisPP13}
F(g) = \sum_{t \in \RR} a_{t}(F) \, \EisCan(g, \Demi + it),
\end{equation}
on a
\begin{equation}
\label{PSHilbert3}
(F_{1}, F_{2})_{2} =
\Demi \, \sum_{t \in \RR} a_{t}(F_{1}) \, \overline{a_{t}(F_{2})}.
\end{equation}
\item
\label{MSHilb3}
Supposons $n = 2$. Si on \'ecrit les \'el\'ements de $\Toe^{1}(X)$ sous
la forme
\begin{equation}
\label{EisPP12}
F(g) = a_{0}(F) \, \EisCan(g, \Demi) +
2 \sum_{t > 0} a_{t}(F) \, \EisCan(g, \Demi + it),
\end{equation}
on a
\begin{equation}
\label{PSHilbert2}
(F_{1}, F_{2})_{2} = a_{0}(F_{1}) \, \overline{a_{0}(F_{2})} +
2 \sum_{t > 0} a_{t}(F_{1}) \, \overline{a_{t}(F_{2})}.
\end{equation}
\end{enumerate}
\end{proposition}

\begin{proof}
Supposons $n \geq 3$ et posons
$$
F_{t}(g) = \sqrt{2} \, \EisCan(g, \Demi + it) \qquad (t \in \RR).
$$
Il d\'ecoule du lemme \ref{Lemme_NormEisDim3} :
\renewcommand{\theenumi}{\roman{enumi}}
\begin{enumerate}
\item
$(F_{t_{1}}, F_{t_{2}})_{2} = 0$ si $t_{1}$ et $t_{2}$
sont dans $\RR$ et $t_{1} \neq t_{2}$.
\item
$\N{F_{t}}_{2}^{2} = 1$ pour tout $t \in \RR$.
\end{enumerate}
\renewcommand{\theenumi}{\alph{enumi}}
On en d\'eduit \eqref{MSHilb2}. Supposons $n = 2$ et posons
$$
F_{t}(g) = \sqrt{2} \, \EisCan(g, \Demi + it) \ (t > 0), \quad
F_{0}(g) = \EisCan(g, \Demi).
$$
Il d\'ecoule du lemme \ref{Lemme_NormEisDim2} :
\renewcommand{\theenumi}{\roman{enumi}}
\begin{enumerate}
\item
$(F_{t_{1}}, F_{t_{2}})_{2} = 0$ si $t_{1}$ et $t_{2}$
sont dans $\RR_{+}$ et $t_{1} \neq t_{2}$.
\item
$\N{F_{t}}_{2}^{2} = 1$ pour tout $t \in \RR_{+}$.
\end{enumerate}
\renewcommand{\theenumi}{\alph{enumi}}
Si $F$ s'\'ecrit sous la forme \'enonc\'ee, on a
$$
F(g) = a_{0}(F) \, F_{0}(g) + \sqrt{2} \sum_{t > 0} a_{t}(F) \, F_{t}(g),
$$
dont on d\'eduit l'expression du produit scalaire, ce qui d\'emontre
\eqref{MSHilb3}.
\end{proof}

\begin{remarks*}
(i) Cette proposition montre que si $n \geq 3$ par exemple, l'application
$$\psi \mapsto \dint \EisCan(g, \Demi + i t) \, d\widehat{\psi}(t)$$
d\'efinit une isom\'etrie surjective de $\Toe^{1}(\RRm)$ sur
$\Toe^{1}(X)$.

(ii) \`A la place de l'espace $\Toe^{1}(X)$, on pourrait prendre
l'espace des fonctions
$$
\int_{\RR} \EisCan(g, \Demi + i t) \, d\mu(t),
$$
o\`u $\mu$ est une \emph{mesure born\'ee} : la construction qui suit conduit
au m\^eme espace de Hilbert, en vertu du th\'eor\`eme taub\'erien de Wiener.
\end{remarks*}

\begin{proof}[D\'emonstration du th\'eor\`eme \ref{Thm_PoHil}]
On pose
$$
\EuScript{Y} =
\set{\gamma \in \RR}{L(\dfrac{1}{2} + i \gamma, \chi) = 0 \, }, \quad
\EuScript{Y}^{+} = \EuScript{Y} \cap \RRm.
$$
La remarque \ref{Rem_SuppAxe} implique que $F \in \Toe^{1}(X)$ appartient
\`a $\Tor$ si et seulement si
\begin{equation}
\label{ElemV3}
F(g) = \sum_{\gamma \in \EuScript{Y}}
a_{\gamma}(F) \, \EisCan(g, \Demi + i\gamma) \quad (n \geq 3),
\end{equation}
\begin{equation}
\label{ElemV2}
F(g) = a_{0}(F) \, \EisCan(g, \Demi) +
2 \sum_{\gamma \in \EuScript{Y}^{+}}
a_{\gamma}(F) \, \EisCan(g, \Demi + i\gamma) \quad (n = 2).
\end{equation}
avec $a_{0}(F) \neq 0$ si et seulement si $L(\Demi, \chi) = 0$.

Si $F$ est donn\'ee par \eqref{EisPP13}, resp. \eqref{EisPP12}, la
proposition \ref{Prop_FcnPropre} implique :
\begin{equation}
\label{DefOmega3}
4 D F(g) = \sum_{t} \ a_{t}(F) \, (1 + 4 \, t^{2}) \, \EisCan(g,
\Demi + it) \quad (n \geq 3),
\end{equation}
\begin{equation}
\label{DefOmega2}
4 D F(g) = a_{0}(F) \, \EisCan(g, \Demi) +
2 \sum_{t > 0} \ a_{t}(F) \, (1 + 4 \, t^{2}) \,
\EisCan(g, \Demi + it)
 \quad (n = 2).
\end{equation}
On d\'eduit des expressions \eqref{PSHilbert3} et \eqref{PSHilbert2} du
produit scalaire de $\Toe^{1}(X)$ que l'op\'erateur diff\'erentiel
$D$ d\'efinit un op\'erateur sym\'etrique sur $\Toe^{1}(X)$ : 
$$
(D F, F')_{2} =
(F, D F')_{2} \quad (F, F' \in \Toe^{1}(X)).
$$
On note $V(D_{E})$ le compl\'et\'e de
$\Toe^{1}(X)$ pour la norme $\N{D F}_{2}$. Cet espace s'identifie
au sous-espace des $F \in B^{2}(X)$ tels que
$$
\N{4 D F}^{2}_{2} =
\sum_{t} \ \abs{a_{t}(F)}^{2} \, (1 + 4 \, t^{2})^{2} < + \infty
\quad (n \geq 3),
$$
$$
\N{4 D F}^{2}_{2} = \abs{a_{0}(F)}^{2} +
2 \sum_{t > 0} \ \abs{a_{t}(F)}^{2} \, (1 + 4 \, t^{2})^{2} < + \infty
\quad (n = 2).
$$
l'op\'erateur d\'efini par \eqref{DefOmega3} ou \eqref{DefOmega2} et de
domaine $V(D_{E})$ est un op\'erateur auto-adjoint de $B^{2}(X)$.
Il en va de m\^eme pour $D_{\chi}$, encore donn\'e par \eqref{DefOmega2} sur
$V(D_{\chi})$. Toute fonction $F \in \Toe^{1}(X) \cap \Tor$ s'\'ecrit
sous la forme \eqref{ElemV3}, resp. \eqref{ElemV2} ; pour ces fonctions, et
si $z \neq \frac{1}{4} + \gamma^{2}$, on a respectivement
$$
(D_{\chi} - z)^{-1} F(g) =
\sum_{\gamma \in \EuScript{Y}} \
\dfrac{a_{\gamma}(F)}{\frac{1}{4} + \gamma^{2} - z} \,
\EisCan(g, \Demi + i\gamma),
$$
$$
(D_{\chi} - z)^{-1} F(g) =
\dfrac{a_{0}(F)}{\frac{1}{4} - z} \, \EisCan(g, \Demi)
+ 2 \sum_{\gamma \in \EuScript{Y}^{+}} \
\dfrac{a_{\gamma}(F)}{\frac{1}{4} + \gamma^{2} - z} \,
\EisCan(g, \Demi + i\gamma).
$$
On v\'erifie que cette r\'esolvante se prolonge \`a $T^{2}(X)$. Si $n = 2$,
la sym\'etrie $\gamma \mapsto - \gamma$ est une permutation des z\'eros de
$L(\Demi + it, \chi)$ ; on a donc
$$
\Spec D_{\chi} =
\set{\lambda = \frac{1}{4} + \gamma^{2}}{\gamma \in \EuScript{Y} \,},
$$
ce qui termine la d\'emonstration de la premi\`ere partie du th\'eor\`eme.
Supposons $n \geq 3$. Pour voir que toute valeur propre $\lambda \neq 1/4$
est double, on remarque que l'expression du terme constant de
$\EisCan(g, s)$ implique
$$
\EisCan^{0}(p, \Demi + it) \sim \delta_{P}(p)^{\Demi + it}
$$
Supposons que la s\'erie d'Eisenstein et sa conjugu\'ee soient d\'ependantes
:
$$
\EisCan(g, \Demi - it) = \lambda \, \EisCan(g, \Demi + it), \quad
\lambda \in \CC.
$$
Ceci implique $\delta_{P}(p)^{-it} = \lambda \, \delta_{P}(p)^{it}$, et donc $t = 0$.
Dans ce cas on a
$$
\EisCan(g, \Demi) = \EisCan(^t\!g^{-1}, \Demi).
$$
et la valeur propre $\lambda = 1/4$ est simple si elle est pr\'esente.
\end{proof}

\medskip

\subsection*{Trace de la repr\'esentation r\'eguli\`ere} ~\medskip

On note $\mathfrak{A}(G)$ l'alg\`ebre auto-adjointe des fonctions continues
\`a support compact sur $Z(\Ad) \backslash G(\Ad)$ bi-invariantes sous
$\mathbf{K}$. On fait op\'erer $\mathfrak{A}(G)$ sur l'espace des fonctions
continues sur $Z(\Ad) \backslash G(\Ad) / \mathbf{K}$ par la repr\'esentation
r\'eguli\`ere droite
$$R(u)F(x) = \int_{Z(\Ad) \backslash G(\Ad)} F(xy) \, u(y) \, dy.$$
Si $u \in \mathfrak{A}(G)$, on a
\begin{equation}
\label{RepReg0}
R(u)\delta_{P}^{s} = \widetilde{u}(s) \, \delta_{P}^{s},
\end{equation}
o\`u on a introduit la fonction enti\`ere
$$\widetilde{u}(s) = \int_{Z(\Ad) \backslash G(\Ad)} \delta_{P}(x)^{s} \, u(x) \,
dx.$$ On d\'eduit de \eqref{RepReg0} que
\begin{equation}
\label{RepReg1}
R(u)\EisCan(g, s) = \widetilde{u}(s) \EisCan(g, s).
\end{equation}
Si $n = 2$, il s'ensuit
\begin{eqnarray*}
\widetilde{u}(s) \, \EisCan(g, s)
& = &
R(u)\EisCan(g, s)
= 
c(s) \, R(u)\EisCan(g, 1 - s)
\\ & = &
c(s) \, \widetilde{u}(1 - s) \, \EisCan(g, 1 - s)
=
\widetilde{u}(1 - s) \, \EisCan(g, s)
\end{eqnarray*}
Puisque les fonctions $g \mapsto \EisCan(g, s)$ ne sont pas identiquement
nulles, comme on le voit en observant $\EisCan^{0}(g, s)$, il vient
\begin{equation*}
\label{InvWeyl}
\widetilde{u}(1 - s) = \widetilde{u}(s).
\end{equation*}

\begin{corollary}[Formule de trace]
\label{Cor_FormTrace}
L'espace $T^{2}(X)$ est invariant sous $\mathfrak{A}(G)$. Soit $u \in
\mathfrak{A}(G)$. Si $n \geq 3$, on a
\begin{equation}
\label{FTR3}
\Tr (R(u) \mid T^{2}(X)) =
\sum_{\gamma \in \EuScript{Y}} \ \widetilde{u}(\Demi + i\gamma),
\end{equation}
Si $n = 2$, on a
\begin{equation}
\label{FTR2}
\Tr (R(u) \mid T^{2}(X)) =
\varepsilon \, \widetilde{u}(\Demi) + \sum_{\gamma \in \EuScript{Y}^{+}} \
\widetilde{u}(\Demi + i\gamma),
\end{equation}
avec $\varepsilon = 1$ ou $0$ suivant que $L(\Demi, \chi) = 0$ ou non.
\end{corollary}

\begin{proof}
On d\'eduit de \eqref{RepReg1} que l'alg\`ebre $\mathfrak{A}(G)$ op\`ere
sur $B^{2}(X)$. Si $n \geq 3$, les fonctions $\EisCan(g, \Demi +
i\gamma)$ pour $\gamma \in \EuScript{Y}$ forment une base orthonormale de
$T^{2}(X)$, on en d\'eduit \eqref{FTR3}. Si $n = 2$, les fonctions
$\EisCan(g, \Demi)$ (si $L(\Demi, \chi) = 0$) et
$\sqrt{2}
\, \EisCan(g, \Demi + i\gamma)$ pour $\gamma \in \EuScript{Y}^{+}$ forment une
base orthonormale de $T^{2}(X)$, on en d\'eduit \eqref{FTR2}.
\end{proof}

\begin{remark}
Les r\'esultats pr\'ec\'edents prennent en compte les z\'eros des fonctions
$L$ avec la m\^eme multiplicit\'e. Or ces fonctions ont des racines
multiples en g\'en\'eral ; toutefois, ce point de vue semble compatible avec
les \emph{conjectures de simplicit\'e de Serre} \cite[Conj. 8.24.1, p.
324]{Goss} sur les z\'eros des fonctions $L(s, \chi)$, lorsque $\chi$ est un
caract\`ere galoisien. Ce sont les suivantes :

\emph{Soit $\chi$ un caract\`ere irr\'eductible de $\Gal(\bar{\QQ} / \QQ)$.
\begin{enumerate}
\item
Les z\'eros non triviaux de $L(s, \chi)$ sont simples.
\item
Pour que $L(\Demi, \chi) = 0$, il faut et il suffit que $\chi$ soit r\'eel
et que l'\'equation fonctionnelle de $L(s, \chi)$ comporte un signe moins.
\item
Si $\chi \neq \chi'$, les z\'eros de $L(s, \chi)$ sont distincts des
z\'eros de $L(s, \chi')$ (mis \`a part \'eventuellement le point $\Demi$).
\end{enumerate}}
\end{remark}

\appendix

\renewcommand{\theequation}{\Alph{section}.\arabic{equation}}

\section{L'espace de Connes}

Alain Connes a introduit en \cite{Connes3} un espace de distributions qui
est le point de d\'epart de la construction de ses espaces de
P\'olya-Hilbert. Nous allons faire le lien avec les espaces de formes
toro\"\i dales. Nous reprenons les calculs de \cite[Eq.(14) et (15), p.
83]{Connes3}. On note $P(\RRm)$ l'espace des fonctions continues sur $\RRm$
qui s'\'ecrivent
\begin{equation*}
\label{DefK0}
\psi(x) = \dint x^{it} \, d\widehat{\psi}(t),
\end{equation*}
o\`u $\widehat{\psi}$ appartient \`a l'espace $\EuScript{E}'(\RR)$ des
distributions \`a support compact sur $\RR$, en notant les distributions
comme des mesures ; la distribution $\widehat{\psi}$ est \emph{la
transform\'ee de Fourier-Mellin} de $\psi$. Si $\psi$ et $\widehat{\psi}$
sont des fonctions int\'egrables, on a
$$
\widehat{\psi}(t) =
\dfrac{1}{2 \pi} \int_{0}^{\infty} x^{-it} \, \psi(x) \, d^{\times}\!x,
$$
et la formule de Parseval s'\'ecrit
$$
\dfrac{1}{2 \pi} \int_{0}^{\infty} \phi(x) \, \psi(x) \,
d^{\times}\!x =
\dint \widehat{\phi}(-t) \, d\widehat{\psi}(t).
$$
Supposons $\varphi \in \EuScript{S}(\Ad^{n})_{0}$, c'est-\`a-dire
$\varphi(0) = \int \varphi = 0$. La fonction $\Th{\varphi}$ est \`a d\'e\-crois\-sance
rapide sur $T(\Ad)$ et l'int\'egrale \eqref{Eisenstein} d\'efinissant
$\Eis{\varphi}(g,s)$ converge pour tout $s \in \CC$.  Un \emph{train d'ondes
d'Eisenstein g\'en\'eral} est une fonction d\'efinie sur $G(k) \backslash
G(\Ad)$ qui s'\'ecrit
$$
W(\psi,\varphi)(g) = \dint \, \Eis{\varphi}(g,\Demi + it) \, d\widehat{\psi}(t),
$$
o\`u $\varphi \in \EuScript{S}(\Ad^{n})_{0}$ et o\`u $\psi \in P(\RRm)$.
L'\'equation \eqref{pgcd} implique
$$
W(\psi,\varphi)(g) = n \dint
\, \EisCan(g, \Demi + it)
\, \Tate_{k}(\boldsymbol{\phi}_{1},n(\Demi + it))
\, d\widehat{\psi}(t),
$$

\begin{lemma}
\label{Lemme_FcnI}
Si $\varphi \in \EuScript{S}(\Ad^{n})_{0}$, si $\psi \in P(\RRm)$, et si $g \in
G(\Ad)$, on a
$$
\EspScal{\psi}{\varphi}(g) =
\int_{Z(k) \backslash Z(\Ad)} \Th{\varphi}(zg) \, \psi(\abs{\det zg}) \, dz.
$$
\end{lemma}

\begin{proof}
Puisque $\varphi \in \EuScript{S}(\Ad^{n})_{0}$, le membre de droite est
convergent, et on a
\begin{eqnarray*}
\EspScal{\psi}{\varphi}(g)
& = &
\dint \, \Eis{\varphi}(g,\Demi + it) \, d\widehat{\psi}(t)
\\ & = &
\int_{Z(k) \backslash Z(\Ad)} \Th{\varphi}(zg) \, dz \dint \, \abs{\det zg}^{it}
\, d\widehat{\psi}(t)
\\ & = &
\int_{Z(k) \backslash Z(\Ad)} \Th{\varphi}(zg) \, \psi(\abs{\det zg}) \, dz.
\end{eqnarray*}
\end{proof}

\begin{lemma}
\label{ProdScal0}
Si $\varphi \in \EuScript{S}(\Ad^{n})_{0}$ et si $\psi \in P(\RRm)$, et $\chi
\in \mathsf{X}$, on a
$$
\oint \EspScal{\psi}{\varphi}(h) \, \chi_{\pi}(h) \, dh =
\int_{T(k) \backslash T(\Ad)} \, \Th{\varphi}(h) \, \psi(\abs{\det h}) \,
\chi_{\pi}(h) \, dh.
$$
\end{lemma}

\begin{proof}
En vertu de \eqref{MesInv}, on a :
\begin{multline*}
\int_{T(k) \backslash T(\Ad)} \Th{\varphi}(h) \, \psi(\abs{\det h}) \,
\chi_{\pi}(h) \, dh
\\ =
\int_{S(k) \backslash S(\Ad)} \, \int_{Z(k) \backslash Z(\Ad)}
\Th{\varphi}(zh) \, \psi(\abs{\det zh}) \, dz \, \chi_{\pi}(zh) \, dh,
\end{multline*}
et on applique le lemme \ref{Lemme_FcnI}.
\end{proof}

\begin{lemma}
\label{Lemme_ProdScal}
Si $\varphi \in \EuScript{S}(\Ad^{n})_{0}$ et si $\psi \in P(\RRm)$, on a
$$
\int_{T(k) \backslash T(\Ad)} \, \Th{\varphi}(h) \, \psi(\abs{\det h}) \,
\chi_{\pi}(h) \, dh =
\dint
\Tate_{K}(\boldsymbol{\phi}_{1},\Demi + it, \chi) \, d\widehat{\psi}(t).
$$
\end{lemma}

\begin{proof}
On applique le lemme \ref{ProdScal0} et la formule de Hecke \eqref{Hecke1}.
\end{proof}

Rappelons que
$$
\EuScript{Y} = \set{\gamma \in \RR} {L(\Demi + i \gamma, \chi) = 0},
$$
et notons $v_{\gamma}$ l'ordre du z\'ero de la fonction $L(\tfrac{1}{2} + i
t, \chi)$ au point $\gamma$. On note $\EuScript{E}'(\EuScript{Y})$ l'espace
des distributions \`a support fini dans $\EuScript{Y}$ de la forme
$$
\sum_{\gamma \in \EuScript{Z}} \, \sum_{j = 0}^{v_{\gamma} -
1} a_{\gamma, j} \, \delta_{\gamma}^{(j)} \qquad (a_{\gamma} \in \CC).
$$

\begin{theorem}[Connes]
\label{Thm_Connes}
Soit $\psi \in P(\RRm)$, et posons
$$
\eta(h) = \chi_{\pi}(h) \, \psi(\abs{\det h}) = \chi_{\pi}(h) 
\dint \abs{\det h}^{it} \, d\widehat{\psi}(t).
$$
Les conditions suivantes sont
\'equivalentes :
\begin{enumerate}
\item
\label{Dist1}
Pour toute fonction $\varphi \in \EuScript{S}(\Ad)_{0}$, on a
$$
\int_{T(k) \backslash T(\Ad)} \, \Th{\varphi}(h) \, \eta(h) \, dh = 0.
$$
\item
\label{Dist2}
Pour toute fonction $\varphi \in \EuScript{S}(\Ad)_{0}$, et pour tout $x \in
T(\Ad)$, on a
$$
\Th{\varphi} * \eta (x)
=
\int_{T(k) \backslash T(\Ad)}
\Th{\varphi}(x^{-1}h) \, \eta(h) \, dh = 0.
$$
\item
\label{Dist3}
La distribution $\widehat{\psi}$ appartient \`a
$\EuScript{E}'(\EuScript{Y})$.
\end{enumerate}
\end{theorem}

\begin{proof}
Le lemme \ref{Lemme_ProdScal} et la formule de Weil \eqref{Tate} impliquent
$$
\int_{T(k) \backslash T(\Ad)} \, \Th{\varphi}(h) \, \eta(h) \, dh
= c_{K}^{-1} \dint
L(\Demi + it, \chi) \, \Tate'_{K}(\boldsymbol{\phi}_{1},\Demi + it, \chi)
\, d\widehat{\psi}(t).
$$
Si $\varphi \in \EuScript{S}(\Ad)_{0}$, on a
$$
\Th{\varphi} * \eta (x) = \int_{T(k) \backslash T(\Ad)}
\Th{\varphi}(x^{-1} h) \, \eta(h) \, dh \, ;
$$
on en d\'eduit, toujours d'apr\`es le lemme \ref{Lemme_ProdScal} :
$$
\Th{\varphi} * \eta(x) = \dint
\Tate_{K}(\widehat{\varphi}_{x^{-1}},\Demi + it, \chi)
\, d\widehat{\psi}(t).
$$
Mais
$$
\Tate_{K}(\widehat{\varphi}_{x^{-1}},\Demi + it, \chi)
=
\chi_{\pi}(x) \, \abs{\det x}^{\Demi + it} \, 
\Tate_{K}(\widehat{\varphi}_{1},\Demi + it, \chi).
$$
On a donc
$$
\Th{\varphi} * \eta(x) = 
c_{K}^{-1} \, \chi_{\pi}(x) \, \abs{\det x}^{1/2}
\dint \abs{\det x}^{it} L(\Demi + it, \chi) \,
\Tate'_{K}(\widehat{\varphi}_{1},\Demi + it, \chi) \,
d\widehat{\psi}(t).
$$
Lorsque $\varphi$ parcourt $\EuScript{S}(\Ad)_{0}$, les fonctions
$\Tate'_{K}(\widehat{\varphi}_{1},\Demi + it, \chi)$ sont denses dans
l'espace
$C^{\infty}(\RR)$ : voir \cite[Eq. (23), p. 85]{Connes3}. Si on introduit
la condition
\begin{enumerate}
\setcounter{enumi}{3}
\item
\label{Dist4}
\emph{La distribution $L(1/2 + it, \chi) \, \widehat{\psi}(t)$ est nulle,}
\end{enumerate}
ce qui pr\'ec\`ede montre que les conditions \eqref{Dist1},
\eqref{Dist2} et \eqref{Dist4} sont \'equivalentes. Enfin, il est bien connu que
les conditions \eqref{Dist3} et \eqref{Dist4} sont \'equivalentes.
\end{proof}

L'espace que Connes utilise pour construire des espaces de P\'olya-Hilbert
est l'espace $\EuScript{H}_{0}$ des fonctions $\eta$ comme ci-dessus, o\`u
$\widehat{\psi}$ appartient \`a $\EuScript{E}'(\EuScript{Y})$. Le lien que
l'on peut faire entre l'espace de Connes et les formes toro\"\i dales est le
suivant. Le lemme
\ref{ProdScal0} et le th\'eor\`eme \ref{Thm_Connes} entra\^{\i}nent :

\begin{corollary}
\label{Cor_Isom}
Soit $\psi \in P(\RRm)$. Les conditions suivantes sont
\'equivalentes :
\begin{enumerate}
\item
\label{ISG1}
Quelle que soit $\varphi \in \EuScript{S}(\Ad^{n})_{0}$, le train d'ondes
$W(\psi,\varphi)$ est toro\"\i dal en $\chi$.
\item
\label{ISG2}
La distribution $\widehat{\psi}$ appartient \`a
$\EuScript{E}'(\EuScript{Y})$.
\hfill \qed
\end{enumerate}
\end{corollary}

\section{S\'eries orbitales et formes toro\"\i dales}

\subsection*{Paraboliques maximaux}~

\medskip

Avec les d\'efinitions de la section \ref{sec_Periodes}, on a un diagramme
$$
\begin{CD}
\mathbf{V} &  @>{\iota}>> & \mathcal{T}_{K/k} \\
@A{j}AA && @VVV & \\
\mathbf{A}_{k}^{n} & @>{\iota}>> & \mathcal{A}_{K/k} \\
\end{CD}
$$
o \`u $\mathbf{V}$ est un ouvert de Zariski de $\mathbf{A}_{k}^{n}$.

On pose ici
$$
P = P_{0} = 
\left\{ \begin{bmatrix} t & 0 \\ u & g' \\ \end{bmatrix} \right\} \, , 
\quad
L = L_{0} = 
\left\{ \begin{bmatrix} 1 & 0 \\ u & g' \\ \end{bmatrix} \right\} \, , 
$$
et si $p \in P$ est \'ecrit comme indiqu\'e, on a ${^{t}\!p}.e_{1} = t e_{1}$.
Introduisons le morphisme 
$$
\begin{CD}
s :& G & @>>> & \mathbf{A}^{n} \setminus \Zero,  & \quad &
s(g) = {^{t}g^{-1}}.e_{1}.
\end{CD}
$$
Soit $G'$ l'image du morphisme $T \times L \longrightarrow G$. Le diagramme
$$
\begin{CD}
\mathcal{T}_{K/k} & @>{\pi}>> & G \\
@A{\iota}AA && @VV{s}V & \\
\mathbf{V} &  @>\mathrm{id}>> & \mathbf{A}^{n}
\setminus \{0\} & \ = G / L \\
\end{CD}
$$
montre d'une part que la restriction de $s$ \`a $T$ est injective,
donc $T \cap L = \{1\}$, et d'autre part que $s(T) = \mathbf{V}$, d'o\`u on
tire que $s^{- 1}(\mathbf{V}) = G'$ est un ouvert de Zariski de $G$.

Soient $\xi$ un \'el\'ement primitif de $K$ et $\gamma = \pi(\xi)$.
Le morphisme $x \mapsto x^{-1} \gamma x$ de $G$ dans $G$ se factorise en
$$
\begin{CD}
G & @>{\dot{x}}>> & T \backslash G & @>{f^{\gamma}}>> G
\end{CD}
$$
o\`u $x \mapsto \dot{x}$ est l'application canonique de $G$ sur $T
\backslash G$, et o\`u $f^{\gamma}(\dot{x}) = \dot{x}^{-1} \gamma \dot{x}$.

\begin{lemma}
\label{Propre}
Pour toute place $v$, l'application injective
$$
\begin{CD}
f^{\gamma} : T(k_{v}) \backslash G(k_{v}) & @>>> G(k_{v})
\end{CD}
$$
est propre. Si $v$ est non-archim\'edienne, on a $(f^{\gamma})^{-1}(U_{v})
\subset \dot{U}_{v}$ pour toute place $v$ et $(f^{\gamma})^{-1}(U_{v}) =
\dot{U}_{v}$ pour presque toute place $v$.
\end{lemma}

\begin{proof}
D'apr\`es le \emph{th\'eor\`eme de compacit\'e de Harish-Chandra}, si
$\omega$ est une partie compacte de $G(k_{v})$, il existe une partie
compacte $\Omega$ de $T(k_{v}) \backslash G(k_{v})$ telle que
$$x^{-1} \gamma x \in \omega \Longrightarrow \dot{x} \in \Omega.$$
Voir \cite[Thm. 8.1.4.1, p. 75]{Warner} si $v$ est r\'eelle ou complexe et
\cite[Lem. 19, p. 52]{Harish2} si $v$ est non-archim\'edienne. Cela
signifie que $f^{\gamma}$ est propre.

Supposons maintenant $v$ non-archim\'edienne. 

Affirmer que $(f^{\gamma})^{-1}(U_{v}) \subset \dot{U}_{v}$ revient \`a dire
que si $x \in G(k_{v})$ et si $x^{-1} \gamma x \in U_{v}$, alors $x \in
T(k_{v}).U_{v}$.

Soit $x \in G(k_{v})$. Si $x^{-1} \gamma x \in U_{v}$, alors $x^{-1} \pi(K) x
\subset \mathbf{M}(n,\mathfrak{o}_{v})$ car $\pi(K) = k[\gamma]$. 

Soit $x = h m \in G'(k_{v})$ avec $h \in T(k_{v})$ et $m
\in L(k_{v})$. Soit $u \in
\mathfrak{o}^{n}$, posons $\eta = \iota(u) \in \mathfrak{O}$ et $\kappa =
x^{-1} \pi(\eta) x = m^{-1} \pi(\eta) m \in \mathbf{M}(n, \mathfrak{o}_{v})$.
On a
$$
{^{t}}\!\kappa.e_{1} = 
{^{t}}\!m.^{t}\!\pi(\eta).^{t}\!m^{-1}.e_{1} = 
{^{t}}\!m.^{t}\!\pi(\eta).e_{1} = {^{t}}\!m.u,
$$
ce qui implique que ${^{t}}\!m.u \in \mathfrak{o}^{n}$ si $u \in
\mathfrak{o}^{n}$, autrement dit $m \in U_{v}$ et $\dot{x} \in \dot{U}_{v}$.
On a d\'emontr\'e que
$$
[(f^{\gamma})^{-1}(U_{v})] \cap [T(k_{v}) \backslash G'(k_{v})]
\subset \dot{U}_{v},
$$
d'o\`u la premi\`ere assertion, car $T(k_{v}) \backslash G'(k_{v})$ est dense
dans $T(k_{v}) \backslash G(k_{v})$, et ensuite $(f^{\gamma})^{-1}(U_{v})$ et
$\dot{U}_{v}$ sont compacts. D'autre part $\gamma \in U_{v}$  et donc
$f^{\gamma}(\dot{U}_{v})
\subset U_{v}$ pour presque toute place $v$. Puisque $f^{\gamma}$ est
injective, ceci implique $\dot{U}_{v} \subset (f^{\gamma})^{-1}(U_{v})$.
\end{proof}

\medskip

\subsection*{Int\'egrales orbitales}~

\medskip

Dans ce qui suit, la lettre $H$ d\'esigne le centre $Z$ ou le
sous-groupe
$T$ de $G$. Ainsi $H_{\Ad} \backslash G_{\Ad} = (H \backslash G)(\Ad)$, et
$\vol \dot{U}_{v} = 1$ pour presque toute place $v$. On note
$C^{\infty}_{c}(H_{\Ad} \backslash G_{\Ad})$ l'espace des combinaisons
lin\'eaires de l'ensemble $B(H_{\Ad} \backslash G_{\Ad})$ des fonctions $u$
sur $H_{\Ad} \backslash G_{\Ad}$ de la forme
$$u(x) = \prod_{v} u_{v}(x_{v})$$
o\`u les fonctions $u_{v}$ sont invariantes \`a gauche par $H(k_{v})$ et
satisfont les conditions suivantes.
\begin{enumerate}
\item
\label{Adm1}
Si $v$ est archim\'edienne, la fonction $u_{v}$ est ind\'efiniment
diff\'erentiable sur $G(k_{v})$ et \`a support compact modulo $H(k_{v})$.
\item
\label{Adm2}
Si $v$ est non-archim\'edienne, la fonction $u_{v}$ est localement
constante sur $G(k_{v})$ et \`a support compact modulo $H(k_{v})$.
\item
\label{Adm3}
pour presque toute place non-archim\'edienne $v$, la fonction $u_{v}$
est la fonction caract\'eristique de $H(k_{v}) U_{v}$.
\end{enumerate}

\begin{proposition}
\label{CompInv}
Si $u \in C^{\infty}_{c}(Z_{\Ad} \backslash G_{\Ad})$, alors
$u^{\gamma} = u \, _{\circ} \, f^{\gamma} \in C^{\infty}_{c}(T_{\Ad} \backslash
G_{\Ad})$.
\end{proposition}

\begin{proof}
On a
$$
u(x^{-1} \gamma x) = u
\, _{\circ} \, f^{\gamma}(\dot{x}) = u^{\gamma}(\dot{x}).
$$
On peut supposer $u \in B(Z_{\Ad} \backslash G_{\Ad})$, de telle sorte que
$$u^{\gamma}(\dot{x}) = \prod_{v} u_{v}^{\gamma}(\dot{x}_{v}), \quad
\text{o\`u} \ u_{v}^{\gamma} = u_{v} \, _{\circ} \, f^{\gamma}.$$
Puisque $\Supp u_{v}^{\gamma} \subset \overline{f^{-1}(\Supp u_{v})}$,
les propri\'et\'es \eqref{Adm1}, \eqref{Adm2}, et \eqref{Adm3} pour
$u^{\gamma}$ r\'esultent du lemme \ref{Propre}.
\end{proof}

\medskip

\subsection*{S\'eries orbitales}~

\medskip

Soient $K$ une extension de degr\'e $n$ de $k$. On note $\pi$ la
repr\'esentation alg\'ebrique r\'eguli\`ere d\'efinie par une base
fondamentale de $K$, et $T$ le tore repr\'esentant les \'el\'ements
inversibles de $K$. Soit $\xi$ un \'el\'ement primitif de $K$, et $\gamma =
\pi(\xi)$, de telle sorte que $T(k') = k'[\gamma]^{\times}$ pour toute
$k$-alg\`ebre
$k'$. Le centralisateur de $\gamma$ est \'egal \`a $\mathcal{Z}(T) = T$
\cite[p. 175]{Borel3}, et l'application $x \mapsto x^{-1} \gamma x$ induit
un isomorphisme
$$T_{k} \backslash G_{k} \isom \mathfrak{c}(\gamma),$$
o\`u $\mathfrak{c}(\gamma)$ est la \emph{classe de conjugaison} de
$\gamma$. On note $\mathfrak{C}_{K}$ l'ensemble des classes de conjugaison
des \'el\'ements $\gamma$ comme ci-dessus.

Si $ \mathfrak{c} \in \mathfrak{C}_{K}$, la \emph{s\'erie orbitale} de $u
\in C^{\infty}_{c}(Z_{\Ad} \backslash G_{\Ad})$ est :
$$
u_{\mathfrak{c}}(x) =
\sum_{\eta \, \in \, \mathfrak{c}} u(x^{-1} \eta x) =
\sum_{\eta \in T_{k} \backslash G_{k}}
u(x^{-1} \eta^{-1} \gamma \eta x) \quad (\gamma \in \mathfrak{c}).
$$
Pour $x$ variant dans un compact fixe, cette somme est finie. Si $\beta \in
G_{k}$, on a
$$
u_{\mathfrak{c}}(\beta x) =
\sum_{\eta \, \in \, \mathfrak{c}}
u(x^{-1} \beta^{-1} \eta \beta x) =
\sum_{\eta \, \in \, \beta^{-1} \mathfrak{c} \beta}
u(x^{-1} \eta x) =
u_{\mathfrak{c}}(x).
$$

\begin{proposition}
Supposons $k = \QQ$ et $n = 2$. Si $\gamma$ est $\QQ$-elliptique, et lorsque
$u \in C_{c}(Z_{\Ad} \backslash G_{\Ad})$, la fonction $u_{\mathfrak{c}}$
est \`a support compact sur $G_{\QQ} Z_{\Ad} \backslash G_{\Ad}$. 
\end{proposition}

\begin{proof}
Posons
$$
A =
\left\{
\begin{pmatrix}
1 & 0 \\
0 & y \\
\end{pmatrix}
\right\} \ ,
\quad
N =
\left\{
\begin{pmatrix}
1 & 0 \\
* & 1 \\
\end{pmatrix}
\right\} \ .
$$ 
Pour $a \in A$, on pose $H(a) = y$.
Pour $t > 0$, on pose
$$
A(t) = \set{a \in A(\RR)}
{H(a) \geq t, \quad y > 0}.
$$
D'apr\`es la th\'eorie de la r\'eduction, on sait qu'il existe un compact
$\omega_{0} \subset N_{\Ad}$ et un nombre $t_{0} > 0$ tels que si $t \leq
t_{0}$, on ait
$$
G_{\Ad} = Z_{\Ad} G_{\QQ} \, \mathfrak{S}(t), \quad \text{o\`u} \
\mathfrak{S}(t) = \omega_{0} A(t) \mathbf{K}.
$$
Soit $\Omega$ une partie de $G_{\Ad}$ qui est compacte
modulo $Z_{\Ad}$. Il existe un nombre $t_{\omega} \geq 1$ avec la
propri\'et\'e suivante : si $x \in \mathfrak{S}(t_{\omega})$, si $\gamma \in
G_{\QQ}$ et si $x^{-1} \gamma x \in \omega$, alors $\gamma \in P_{\QQ}$
\cite[p. 362]{Arthur0}. Soit $x \in \mathfrak{S}(t)$ et posons $\omega =
\Supp u$. Puisque $\gamma$ est elliptique, aucun \'el\'ement $\eta \in
\mathfrak{c}$ n'appartient  \`a
$P_{\QQ}$. Si $x^{-1} \eta x \in \omega$, ce qui pr\'ec\`ede implique
que $x = n a \kappa$, avec $t \leq H(a) \leq t_{\omega}$, qui est un
compact de $\mathfrak{S}(t)$.
\end{proof}

Rappelons que
$$
\Pi_{1}(F)(x) =
\int_{T_{k}  Z_{\Ad} \backslash T_{\Ad}} \ F(hx) \, dh \, dx.
$$

\begin{theorem}
\label{ThmClasses}
Soit $ \mathfrak{c} \in \mathfrak{C}_{K}$. Si $F \in C(G_{k} Z_{\Ad}
\backslash G_{\Ad})$, et si $u \in C^{\infty}_{c}(Z_{\Ad} \backslash
G_{\Ad})$, l'int\'egrale
$$
J_{\mathfrak{c}}(u, F) =
\int_{G_{k} Z_{\Ad} \backslash G_{\Ad}} \ 
 u_{\mathfrak{c}}(x) \, F(x) \, dx
$$
converge absolument et on a
$$
J_{\mathfrak{c}}(u, F) =
\int_{T_{\Ad} \backslash G_{\Ad}} \ 
u(x^{-1} \gamma x) \, \Pi_{1}(F)(x) \, dx
$$
quel que soit $\gamma \in \mathfrak{c}$.
\end{theorem}

\begin{proof}
On a
\begin{eqnarray*}
J_{\mathfrak{c}}(u, F)
& = &
\int_{G_{k} Z_{\Ad} \backslash G_{\Ad}} \
\sum_{\eta \in T_{k} \backslash G_{k}}
u(x^{-1} \eta^{-1} \gamma \eta x) \, F(x) \, dx
\\ & = &
\int_{T_{k} Z_{\Ad} 
\backslash G_{\Ad}} \  u(x^{-1} \gamma x) \, F(x) \, dx
\\ & = &
\int_{T_{\Ad} \backslash G_{\Ad}} \ 
\int_{T_{k}  Z_{\Ad} \backslash T_{\Ad}} \ 
u(x^{-1} \gamma x) \, F(hx) \, dh \, dx.
\end{eqnarray*}
Mais la fonction $x \mapsto u(x^{-1} \gamma x)$ est invariante \`a gauche
sous $T_{\Ad}$, donc
$$
J_{\mathfrak{c}}(u, F) =
\int_{T_{\Ad} \backslash G_{\Ad}} \ 
u(x^{-1} \gamma x)
\int_{T_{k} Z_{\Ad} \backslash T_{\Ad}} \ 
\, F(hx) \, dh \, dx,
$$
d'o\`u le r\'esultat, l'int\'egrale convergeant gr\^ace \`a la proposition
\ref{CompInv}.
\end{proof}
\begin{corollary}
Pour que $F \in C(X)$ soit toro\"\i dale pour le caract\`ere principal, il faut et
il suffit que l'on ait
$$
J_{\mathfrak{c}}(u, F) =
\int_{G_{k} Z_{\Ad} \backslash G_{\Ad}} u_{\mathfrak{c}}(x) \, F(x) \, dx =
0 \quad \text{pour tout} \ u \in \mathfrak{A}(G).
$$
pour tout $u \in \mathfrak{A}(G)$ et pour toute $\mathfrak{c} \in
\mathfrak{C}_{K}$. Autrement dit, si on note $U_{\mathfrak{c}}(X)$ l'espace
des fonctions
$$
u_{\mathfrak{c}}(x) =
\sum_{\eta \, \in \, \mathfrak{c}} u(x^{-1} \eta x),
$$
o\`u $u$ parcourt $C_{c}^{\infty}(Z_{\Ad} \backslash G_{\Ad} / \mathbf{K})$,
on a
$C(X) \cap T(X) = C(X) \cap U_{\mathfrak{c}}(X)^{\perp}$, la dualit\'e
\'etant celle de $L^{2}(G_{\Ad})$.
\end{corollary}

\begin{proof}
Si $F$ est toro\"\i dale pour le caract\`ere principal, le th\'eor\`eme
\ref{ThmClasses} implique $J_{\mathfrak{c}}(u, F) = 0$. R\'ecipro\-quement,
si $J_{\mathfrak{c}}(u, F) = 0$ pour tout $u \in C_{c}(Z_{\Ad}
\backslash G_{\Ad})$, on voit que $\Pi_{1}(F) = 0$ en prenant pour $u$ la
mesure de Dirac de $\gamma$ (sic), de telle sorte que $u(x^{-1}
\gamma x)$ est la mesure de Dirac de la classe
$T_{\Ad}$ dans $T_{\Ad} \backslash G_{\Ad}$.
\end{proof}

\begin{remark*}
Ce corollaire s'applique aux trains d'ondes d'Eisenstein : on a
$$E^{1}(X) \cap T(X) = E^{1}(X) \cap U_{\mathfrak{c}}(X)^{\perp}$$
pour toute classe de conjugaison comme ci-dessus.
\end{remark*}

\begin{corollary}
\label{ClassesEisGen}
Si $\varphi \in \EuScript{S}(\Ad^{n})$, posons
$$
J_{\mathfrak{c}}(\varphi)(u, s) =
\int_{G_{k} Z_{\Ad} \backslash G_{\Ad}} \ 
u_{\mathfrak{c}}(x) \, \Eis{\varphi}(x,s) \, dx.
$$
Alors la fonction $\zeta_{K}(s)$ divise $J_{\mathfrak{c}}(\varphi)(u, s)$
: on a
$$
J_{\mathfrak{c}}(\varphi)(u, s) = \zeta_{K}(s)
\int_{T_{\Ad} \backslash G_{\Ad}} \ u(x^{-1} \gamma x) \,
\abs{\det x}^{s} \Tate'_{K}(\Phi_{x}, s) \, dx,
$$
o\`u $\Phi_{x}(\xi) = \varphi(^{t}\!x^{t}\!\pi(\xi).e_{1}) \in
\EuScript{S}(\Ad_{K})$. Si $\Re(s) > 1$, on a
$$
J_{\mathfrak{c}}(\varphi)(u, s) =
\int_{G_{\Ad}} \ 
u(x^{-1} \gamma x) \, \varphi({^{t}x}.e_{1}) \, \abs{\det x}^{s} \, dx.
$$
\end{corollary}

\begin{proof}
Posons $F(x) = \Eis{\varphi}(x,s)$. La formule de Hecke implique
$$
\int_{T_{k}  Z_{\Ad} \backslash T_{\Ad}} \ \Eis{\varphi}(hx,s) \, dh
= \zeta_{K}(s) \, \abs{\det x}^{s} \Tate'_{K}(\Phi_{x}, s) \, ; 
$$
on applique le th\'eor\`eme \ref{ThmClasses}.
D'autre part, si $\Re(s) > 1$, on a
$$
\Eis{\varphi}(g,s) =
\sum_{\gamma \in P_{k} \backslash G(k)} M(\varphi)(\gamma x,s).
$$
Par cons\'equent
$$
J_{\mathfrak{c}}(\varphi)(u, s) =
\int_{P_{k} Z_{\Ad} \backslash G_{\Ad}} \ 
u_{\mathfrak{c}}(x) \, M(\varphi)(x,s) \, dx.
$$
Mais $L_{k} = Z_{k} \backslash P_{k}$ est un syst\`eme de repr\'esentants de
$T_{k} \backslash G_{k}$, et on a donc
$$
u_{\mathfrak{c}}(x) =
\sum_{\eta \in Z_{k} \backslash P_{k}} u(x^{-1} \eta^{-1} \gamma \eta x).
$$
Puisque $Z_{k} = P_{k} \cap Z_{\Ad}$, on a
$(Z_{k} \backslash P_{k}) \backslash (Z_{\Ad} \backslash G_{\Ad}) =
P_{k} Z_{\Ad} \backslash G_{\Ad}$, et donc
$$
J_{\mathfrak{c}}(\varphi)(u, s) =
\int_{Z_{\Ad} \backslash G_{\Ad}} \ 
u(x^{-1} \gamma x) \, M(\varphi)(x,s) \, dx.
$$
Mais
$$
M(\varphi)(x,s) = \int_{Z_{\Ad}} \varphi(z.{^{t}x}.e_{1}) \,
\abs{\det z x}^{s} dz
$$
d'o\`u le r\'esultat.
\end{proof}

La \emph{transform\'ee d'Eisenstein} de $f \in C(X)$ est
$$
\widetilde{f}(s) = \int_{G_{k} Z_{\Ad} \backslash G_{\Ad}}
f(x) \, \mathbf{E}(x,s) \, dx,
$$
lorsque l'int\'egrale converge absolument.

\begin{corollary}
\label{ClassesEisNorm}
Soit $u \in \mathfrak{A}(G)$. La fonction $\widetilde{u}_{\mathfrak{c}}(s)$
est m\'eromorphe, et $\zeta_{K}(s)$ divise $\widetilde{u}_{\mathfrak{c}}(s)$
: si $s$ n'est pas un p\^ole de
$\mathbf{E}(x,s)$, alors
$$
\widetilde{u}_{\mathfrak{c}}(s) =
\zeta_{K}(s) \, \int_{T_{\Ad} \backslash G_{\Ad}} \
u(x^{-1} \gamma x) \, H(x,s,1) \, dx.
$$
De plus
$$
\widetilde{u}_{\mathfrak{c}}(s) =
\int_{Z_{\Ad} \backslash G_{\Ad}} \ 
u(x^{-1} \gamma x) \, \delta_{P}(x)^{s} \, dx.
$$
\end{corollary}

\begin{proof} 
On peut d\'eduire ce corollaire du corollaire \ref{ClassesEisGen} ; voici
une d\'e\-mons\-tra\-tion directe.  La premi\`ere formule vient du
th\'eor\`eme
\ref{ThmClasses} et de la relation
$$
\int_{T_{k}  Z_{\Ad} \backslash T_{\Ad}} \ \mathbf{E}(hx, s) \, dh =
\zeta_{K}(s) \, H(x,s,1).
$$
D'autre part, si $\Re(s) > 1$, on a
$$
\mathbf{E}(x,s) =
\sum_{\gamma \in P_{k} \backslash G(k)} \delta_{P}(\gamma x)^{s}.
$$
Par cons\'equent
$$
\widetilde{u}_{\mathfrak{c}}(s) =
\int_{P_{k} Z_{\Ad} \backslash G_{\Ad}} \ 
u_{\mathfrak{c}}(x) \, \delta_{P}(x)^{s} \, dx.
$$
Mais $Z_{k} \backslash P_{k}$ est un syst\`eme de repr\'esentants de $T_{k}
\backslash G_{k}$, et on a donc
$$
u_{\mathfrak{c}}(x) =
\sum_{\eta \in Z_{k} \backslash P_{k}} u(x^{-1} \eta^{-1} \gamma \eta x),
$$
d'o\`u la deuxi\`eme formule, d'abord pour $\Re(s) > 1$ et par prolongement
analytique en g\'en\'eral. 
\end{proof}

\end{document}